\documentclass{amsart}
\usepackage[paper]{hzmath}

\makeatletter
\@namedef{subjclassname@2020}{%
  \textup{2020} Mathematics Subject Classification}
\makeatother

\begin{document}

\title[uniform propagation minimax]{Uniform-in-time propagation of chaos for consensus-based minimax algorithm}

\author[Bayraktar]{Erhan Bayraktar} \thanks{E. Bayraktar is supported in part by the National Science Foundation under grant DMS-2507940 and by the Susan M. Smith Chair.}
\address{%
	Department of Mathematics,
	University of Michigan,
	Ann Arbor, MI 48109.}
\email{erhan@umich.edu  }

\author[Ding]{Zhiyan Ding}
\address{%
	Department of Mathematics,
	University of Michigan,
	Ann Arbor, MI 48109.}
\email{zyding@umich.edu  }

\author[Ekren]{Ibrahim Ekren}
\address{%
	Department of Mathematics,
	University of Michigan,
	Ann Arbor, MI 48109.}
\email{iekren@umich.edu  }
\thanks{I.~Ekren is supported in part by the National Science Foundation under grant DMS-2406240.}

\author[Zhou]{Hongyi Zhou}
\address{%
	Department of Mathematics,
	University of Michigan,
	Ann Arbor, MI 48109.}
\email{hongyizh@umich.edu	}

\date{}

\begin{abstract}
	We study the large-population convergence of a consensus-based algorithm for the saddle point problem proposed by~\cite{HuangQiuRiedl2024}, establishing the uniform-in-time propagation of chaos using a coupling method.
    Our work shows that the $L^2$-deviation has order $O(N_1^{-1} + N_2^{-1})$ uniformly in time, where $N_1$ and $N_2$ denote the numbers of particles corresponding to the two competing players.
    It demonstrates the convergence of the particles to some location near a saddle point of the given objective function, which confirms the computational feasibility of the algorithm.
    The main idea behind the proofs is the exponential decay and the concentration of the variances of the particle system. 
\end{abstract}

\maketitle

\tableofcontents

\section{Introduction}

The \emph{minimax} (or \emph{upper-value}) problem aims to find the best option under the worst-case scenario.
It can be interpreted as the \emph{upper value} of a two-player zero-sum game,
\begin{equation}
\label{eq:minimax}
    \overline V := \min_{x \in \cX} \max_{y \in \cY} \cE(x,y) \,,
\end{equation}
where $\cX$ and $\cY$ are the sets of options for the two competing players, and $\cE: \cX \times \cY \to \R$ is the cost function.
Dually, the \emph{lower value} of the game is defined by
\[
\underline V := \max_{y \in \cY} \min_{x \in \cX} \cE(x,y),
\]
and in general one only has $\underline V \le \overline V$.
A pair $(x^\ast,y^\ast)$ is called a \emph{saddle point} (or \emph{Nash equilibrium}) if
\[
\cE(x^\ast,y)\le \cE(x^\ast,y^\ast)\le \cE(x,y^\ast),
\qquad \forall x\in\cX,\ \forall y\in\cY.
\]
In this case, $\underline V=\overline V=\cE(x^\ast,y^\ast)$, and the zero-sum game admits a value.

Abundant precedent works, including~\cite{Sion1958,vonNeumann1928}, have studied the existence of solutions under certain structural properties of the objective function.
One may find a more detailed background of the theory in~\cite{vonNeumannMorgenstern2007,Myerson1997}.
The theory of zero-sum games admits various applications in biological sciences, economics, engineering, social sciences, etc.
In recent years, minimax optimization has also attracted significant attention within the machine learning community; see~\cite{ChangHongWaiZhangLu2020,GoharyHuangLuoPang2009,GoodfellowPouget-AbadieMirzaXu2020,LiuDaiLuo2013,MadrasCreagerPitassiZemel2018,MadryMakelovSchmidtTsiprasVladu2018,OmidshafieiPazisAmatoHowVian2017}.
The demand for computational algorithms to find solutions thus keeps pushing forward the study of this subject.

Similar to the convex optimization methods, some variants of the gradient descent-ascent (GDA) algorithms (described in detail by~\cite{Bubeck2015,Hazan2016,RazaviyaynHuangLuNouiehed2020}) were proposed to solve the minimax problem in the convex-concave setting and even convex-or-concave setting~\cite{NouiehedSanjabiHuangLee2019}.
The GDA algorithms are effective and efficient under those circumstances, but the problem becomes NP-hard when the objective function is non-convex-concave, see~\cite{MurtyKabadi1987}.
Some recent works~\cite{DaskalakisPanageas2018,MazumdarJordanSastry2019finding,LiuRafiqueLinYang2021,LoizouBerardJolicoeur-Martineau2020} devise methodologies to find local or approximated Nash equilibria under milder conditions on the objective function. 

In this paper, we study the behavior of a population-based algorithm proposed by~\cite{HuangQiuRiedl2024}.
In fact, earlier in this century, there were already several works on minimax optimization algorithms that make use of the crowding effect of particle systems, including~\cite{KrohlingHoffmannCoelho2004,LaskariParsopoulosVrahatis2002,ShiKrohling2002}, which were mainly inspired by particle swarm optimization (PSO)~\cite{KennedyEberhart1995}.
Joining the idea of consensus-based optimization (CBO)~\cite{PinnauTotzeckTseMartin2017,CarriloChoiTotzeckTse2018}, the work~\cite{HuangQiuRiedl2024} proposes a gradient-free algorithm to find approximate solutions to the minimax problem on $\cX = \R^{d_1}$ and $\cY = \R^{d_2}$. 

Technically, the consensus-based minimax (CBM) algorithm is given by the particle system
\begin{align*}
    \bX^{N_1} = (X^{N_1,1}, \dots, X^{N_1,N_1}) \in \R^{N_1 d_1},\quad \bY^{N_2}  = (Y^{N_2,1}, \dots, Y^{N_2,N_2}) \in \R^{N_2 d_2}
\end{align*}
under the dynamics
\begin{subequations}
\label{eq:algorithm}
    \begin{equation}
        dX^{N_1,k_1}_t = -\lambda_1 (X^{N_1,k_1}_t - \cM_\alpha(\nu_{\bX^{N_1}_t}, \nu_{\bY^{N_2}_t}) )dt + \sigma_1 D_1(X^{N_1,k_1}_t, \cM_\alpha(\nu_{\bX^{N_1}_t}, \nu_{\bY^{N_2}_t})) dW^{X,k_1}_t
    \end{equation}
    for each $k_1 \in [N_1]$,
    \begin{equation}
        dY^{N_2,k_2}_t  = -\lambda_2 (Y^{N_2,k_2}_t - \cM_{-\beta}(\nu_{\bY^{N_2}_t}, \nu_{\bX^{N_1}_t}) )dt + \sigma_2 D_2(Y^{N_2,k_2}_t, \cM_{-\beta}(\nu_{\bY^{N_2}_t}, \nu_{\bX^{N_1}_t})) dW^{Y,k_2}_t
    \end{equation}
    for each $k_2 \in [N_2]$, where $\lambda_1, \lambda_2, \sigma_1, \sigma_2$ are positive constants, $\{W^{X,k_1}\}_{k_1 \in \N_+}$ and $\{W^{Y,k_2}\}_{k_2 \in \N_+}$ are independent Brownian motions,
    and the initial state is given by
    \begin{equation}
        \bX^{N_1}_0 \sim (\bar \nu^X_0)^{\otimes N_1} \,, \qquad \bY^{N_2}_0 \sim (\bar \nu^Y_0)^{\otimes N_2} 
    \end{equation}
    for some probability measures $\bar\nu^X_0$ and $\bar\nu^Y_0$.
\end{subequations}
Here $\nu_{\cdot}$ denote the empirical measures
\begin{equation*}
    \nu_{\bX^{N_1}_t} = \frac{1}{N_1} \sum_{j=1}^{N_1} \delta_{X^{N_1,j}_t} \,, \qquad \nu_{\bY^{N_2}_t} = \frac{1}{N_2} \sum_{j=1}^{N_2} \delta_{Y^{N_2,j}_t} \,,
\end{equation*}
and $\cM$ is the consensus operator defined for by
\begin{align*}
    \cM_\alpha(\mu_1, \mu_2)  = \frac{\int_\cX x e^{-\alpha \cE(x, \ip{y, \mu_2})} \mu_1(dx)}{\int_\cX e^{-\alpha \cE(x, \ip{y, \mu_2})}\mu_1(dx)} \,,\qquad \cM_{-\beta}(\mu_2, \mu_1)  = \frac{\int_\cY y e^{\beta \cE(\ip{x,\mu_1}, y)} \mu_2(dy)}{\int_\cY e^{\beta \cE(\ip{x,\mu_1}, y)} \mu_2(dy)} \,,
\end{align*}
for $\alpha, \beta \ge 0$, $\mu_1 \in \cP(\R^{d_1})$, $\mu_2 \in \cP(\R^{d_2})$, where $\cP(\cdot)$ denotes the set of probability measures on the corresponding measurable space.
In addition, $D_i$ are the anisotropic diffusion with a compactly supported multiplier $\phi_i$, given by
\begin{equation*}
    D_i(z,m) = \diag(z-m) \phi_i(z)
\end{equation*}
for $z, m \in \R^{d_i}$, $i=1,2$.
We will sometimes neglect the subscript $i$ for $D$ when the dimension is clear.

\subsection{Main problem and our contributions}

The work~\cite{HuangQiuRiedl2024} has established the well-posedness of the above system.
When $N_1$ and $N_2$ become large, one may reasonably extrapolate that the particles will behave almost independently and identically, with distributions $(\bar\nu^X_t)_{t \ge 0}$ and $(\bar\nu^Y_t)_{t \ge 0}$ of the $X$- and the $Y$-particles, respectively, being the weak solutions to the mean-field SDEs
\begin{align*}
    d \bar X_t & = -\lambda_1 (\bar X_t - \cM_\alpha(\bar\nu^X_t, \bar\nu^Y_t)) dt + \sigma_1 D(\bar X_t, \cM_\alpha(\bar\nu^X_t, \bar\nu^Y_t)) dW^X_t \,, \\
    d \bar Y_t & = -\lambda_2 (\bar Y_t - \cM_{-\beta}(\bar\nu^Y_t, \bar\nu^X_t)) dt + \sigma_2 D(\bar Y_t, \cM_{-\beta}(\bar\nu^Y_t, \bar\nu^X_t)) dW^Y_t \,.
\end{align*}
It is proven in~\cite{HuangQiuRiedl2024} that these two mean-field processes converge exponentially fast to some limit point close to a saddle point as $t\to\infty$.

In this manuscript, we study the proximity of the particle system~\eqref{eq:algorithm} to its mean-field limit in the $L^2$ sense over an infinite time horizon as $N_1, N_2\to \infty$.
We construct independent copies
\begin{equation*}
    \bar\bX^{N_1} = (\bar X^1, \dots, \bar X^{N_1}) \,, \qquad \bar\bY^{N_2} = (\bar Y^1, \dots, \bar Y^{N_2})
\end{equation*}
of the above mean-field SDEs using the Brownian motions $W^{X,k_1}$ and $W^{Y,k_2}$ identical to those in~\eqref{eq:algorithm}.
This gives a coupling between the empirical distribution $(\nu_{\bX^{N_1}_t}, \nu_{\bY^{N_2}_t})_{t \ge 0}$ of the particle system and that of the samples from the extrapolated limit laws $(\bar\nu^X_t, \bar\nu^Y_t)_{t \ge 0}$, which gives an upper bound on their Wasserstein 2-distance. 
The theory of propagation of chaos suggests that the distance should decay to zero as the size $(N_1, N_2)$ of the particle system tends to infinity.

Classical theory on the propagation of chaos for homogeneous interacting particle systems, which is fully explained in~\cite{Sznitman1991}, shows that the argument generally holds on finite time intervals, subject to some regularity conditions on the drift and diffusion coefficients.
However, as demonstrated in~\cite{HuangQiuRiedl2024}, the result of interest about the CBM mean-field limit processes is in fact the exponential convergence in infinite time.
This naturally raises the question of whether the propagation of chaos still holds when the algorithm runs for an arbitrarily long time. 
In practice, it determines whether the choice of the population size can be independent of the running time.

The expected argument belongs to a specific category, the \emph{uniform-in-time propagation of chaos}.
It is more challenging because of the commonly accumulated errors over time.
The uniform-in-time result thus relies on some sense of ergodicity of the mean-field limit process. 
Some examples are studied in~\cite{Malrieu2001,Malrieu2003,GuillinLebrisMonmarche2023,RosenzweigSerfaty2023,Lacker2021,LackerFlem2023,DelarueTse2025}. 
While CBM algorithms involve two-player dynamics, they utilize particle-concentration terms similar to the original CBO algorithm. 
These terms are instrumental in achieving uniform-in-time propagation of chaos for CBO, as demonstrated by \cite{BayraktarEkrenZhou2025, GerberHoffmannKimVaes2025}; consequently, we anticipate a comparable result for the CBM framework. 
Specifically, we demonstrate that, with appropriate parameters, there exists some constant $C > 0$, independent of $N_1$ and $N_2$, such that
\begin{equation}
\label{eq:goal}
    \sup_{t \ge 0} \E \left[ \frac{1}{N_1} \sum_{k_1=1}^{N_1} \abs{X^{N_1,k_1}_t - \bar X^{k_1}_t}^2 + \frac{1}{N_2} \sum_{k_2=1}^{N_2} \abs{Y^{N_2,k_2}_t - \bar Y^{k_2}_t}^2 \right] \le C (N_1^{-1} + N_2^{-1}) \,.
\end{equation}
This means the empirical distribution $(\nu_{\bX^{N_1}_t}, \nu_{\bY^{N_2}_t})$ is within $O(N_1^{-1/2} + N_2^{-1/2})$-distance, under the Wasserstein 2-metric,
from the samples of its mean-field limit $(\bar \nu^X_t, \bar \nu^Y_t)$ uniformly in time $t \in [0,\infty)$.

To obtain this uniform-in-time upper bound, we must show that the $L^2$-error on the left-hand side of~\eqref{eq:goal}, which we denote by $\E[\cD^X_{N_1} + \cD^Y_{N_2}]$, satisfies the estimate 
\begin{equation}
\label{eq:decay-est}
    \frac{d}{dt} \E[\cD^X_{N_1}(t) + \cD^Y_{D_2}(t)] \le C_{\text{decay}} e^{-\zeta t} \E[\cD^X_{N_1}(t) + \cD^Y_{N_2}(t)] + C_{\text{error}} e^{-\zeta t} (N_1^{-1} + N_2^{-1}) 
\end{equation}
for some constants $C_{\text{decay}}, C_{\text{error}}, \zeta$ that we will explicitly define later.
The core of the proof is the rapid decay of variances. 
We define a generalized notion of variance by 
\begin{equation*}
    \cV_{p}(\mu) \defeq \int \abs{x - M(\mu)}^p \mu(dx) \,, \qquad \text{where } M(\mu) \defeq \int x \mu(dx) \,.
\end{equation*}
Note that $\cV_2$ is the usual variance.
Then we show that the variance $\cV_{2q}$ of the empirical distributions $(\nu_{\bX^{N_1}_t}, \nu_{\bY^{N_2}_t})_{t \ge 0}$ and the mean-field laws $(\bar \nu^X_t, \bar \nu^Y_t)_{t \ge 0}$ decay exponentially fast.

The main obstacle in the proof comes from the quantity
\begin{equation}
\label{eq:obstacle}
    \E \left[ \left( \var{\nu_{\bX^{N_1}_t}} + \var{\nu_{\bar\bX^{N_1}_t}} \right) \left( \cW_2^2(\nu_{\bX^{N_1}_t}, \nu_{\bar\bX^{N_1}_t}) + \cW_2^2(\nu_{\bY^{N_2}_t}, \nu_{\bar\bY^{N_2}_t}) \right) \right] 
\end{equation}
and its counterpart for the $Y$-particles. 
A direct application of Hölder's inequality raises the exponent of $\cD^X_{N_1}$ and $\cD^Y_{N_2}$, which asymptotically overtakes the right-hand side of the estimate~\eqref{eq:decay-est}.
This leads to the demand for a stronger property of the variance.
Thus, upon the convergence, we further establish the concentration inequalities of the form
\begin{align*}
    \P \left[ \sup_{t \ge 0} e^{\kappa t} \var{\nu_{\bX^{N_1}_t}} \ge \var{\bar\nu^X_0} + A \right] \le C_{\text{tail,1}}(q) A^{-q} N_1^{-\frac q2} \cV_{2q}(\bar\nu_0)\,, \\
    \P \left[ \sup_{t \ge 0} e^{\kappa t} \var{\nu_{\bar\bX^{N_1}_t}} \ge \var{\bar\nu^X_0} + A \right] \le \bar C_{\text{tail,1}}(q) A^{-q} N_1^{-\frac q2} \cV_{2q}(\bar\nu_0)\,,
\end{align*}
in the same manner as~\cite{GerberHoffmannKimVaes2025}.
This allows us to split the obstacle quantity~\eqref{eq:obstacle} by the values of $\var{\nu_{\bX^{N_1}_t}}$ and $\var{\nu_{\bar\bX^{N_1}_t}}$, applying the concentration inequalities to bound the part where $\var{\nu_{\bX^{N_1}_t}}$ and $\var{\nu_{\bar\bX^{N_1}_t}}$ are both greater than $e^{-\kappa t} (\cV_2(\bar\nu^X_0) + 1)$.
Then~\eqref{eq:obstacle} still admits an upper bound of order $e^{-\kappa t} (N_1^{-1} + N_2^{-1})$, which aligns with the estimates from other parts of the proof. 

It is worth noticing that our algorithm~\eqref{eq:algorithm} is conducted on a compact domain.
This avoids sudden blowups of runtime data and thus ensures memory safety.
It also allows us to lift the a priori global boundedness condition on the objective function $\cE$; the upper and lower bounds on the objective values of concern can be instead computed by the size of the domain and the growth rate of $\cE$.
However, such an additional multiplicative factor in the diffusion coefficient also leads to nonlinearity so that a variance term does not directly factor out to give an exponential decay, as seen in Step 3, Section~\ref{s:prf-main}. 
This requires another application of the concentration inequalities mentioned above, and it is one of the main technical features that distinguish our work from~\cite{GerberHoffmannKimVaes2025}.

\subsection{Organization of this paper}

The above paragraphs provide a brief layout of the main contents.
We describe the problem setting and preliminary results in detail in Sections~\ref{s:setting} and~\ref{s:poc}, with the main argument stated as Theorem~\ref{th:main}.
In Section~\ref{s:var}, we present the aforementioned core estimates on the variances, with the proofs given in Section~\ref{s:prf-var}.
It then follows the complete proof of the main theorem in Section~\ref{s:prf-main}.
The rest of this paper (Section~\ref{s:tech}) justifies the auxiliary results used in the previous proofs.

\section{Main results}

\subsection{Notation and setting of the algorithm}\label{s:setting}

We take $\cX = \R^{d_1}$ and $\cY = \R^{d_2}$ for some given $d_1, d_2 \in \N_+$ throughout this paper.
However, the solutions of interest typically appear in a predetermined region. 
It is thus reasonable to restrict our particle system to some compact domain $B_{d_1}(0, R_\cut) \times B_{d_2}(0, R_\cut)$ via some cutoff functions $\phi_i$, $i=1,2$, where the existence of a Nash equilibrium is guaranteed.
Here $B_d(c,r)$ denotes the closed ball in the space $\R^d$ equipped with the usual Euclidean 2-norm, with center $c \in \R^d$ and $r > 0$.
The following condition defines the cutoff functions.
\begin{condition}
\label{cd:cutoff}
    The functions $\phi_i \in C^\infty_c(\R^{d_i})$, $i=1,2$, satisfy
    \begin{equation*}
        0 \le \phi_i(z) \le 1 \,, \qquad \forall z \in \R^{d_i}
    \end{equation*}
    and
    \begin{equation*}
        \phi_i(z) = 0 \,, \qquad \forall z \notin B_{d_i}(0, R_\cut)\,,
    \end{equation*}
    for some constant $R_\cut > 0$,
\end{condition}


In real applications, the objective function $\cE$ is usually non-convex-concave, but it often admits some boundedness and smoothness structure. 
For the purpose of quantitative analysis, in this work, we adopt the regularity conditions from~\cite{HuangQiuRiedl2024} with mild changes. 
We note that these conditions are standard to ensure the well-posedness of the stochastic systems.
\begin{condition}
\label{cd:efn}
    There exist constants $L_\cE, C_{\overline{\cE}}, C_{\underline{\cE}}, C_\cE$ such that:

    \begin{enumerate}
        \item The objective function $\cE$ is $C^2$ and locally Lipschitz in the sense that
        \begin{equation*}
            \abs{\cE(x,y) - \cE(x',y')} \le L_\cE (1 + \abs{x} + \abs{x'} + \abs{y} + \abs{y'}) (\abs{x-x'} + \abs{y-y'})
        \end{equation*}
        for all $x,x' \in \R^{d_1}$ and $y,y' \in \R^{d_2}$.

        \item There exist some functions $\overline{\cE}: \R^{d_1} \to \R$ and $\underline{\cE}: \R^{d_2} \to \R$ such that
        \begin{equation*}
            \underline{\cE}(y) \le \cE(x,y) \le \overline{\cE}(x)
        \end{equation*}
        for all $x \in \R^{d_1}$ and $y \in \R^{d_2}$.

        \item The bound functions $\overline{\cE}$ and $\underline{\cE}$ grow at most quadratically in the sense that
        $\overline{\cE}(x) \le C_{\overline{\cE}} (1 + \abs{x}^2)$ for all $x \in \R^{d_1}$ and $\underline{\cE}(y) \ge -C_{\underline{\cE}} (1+\abs{y}^2)$ for all $y \in \R^{d_2}$.

        \item The bound functions $\overline{\cE}$ and $\underline{\cE}$ are not far from the objective $\cE$ itself in the sense that 
        \begin{align*}
            \cE(x,y) - \underline{\cE}(y + sv) & \le C_\cE(1 + \abs{x}^2 + \abs{y}^2 + \abs{v}^2) \,, \\
            \overline{\cE}(x+su) - \cE(x,y) & \le C_\cE(1 + \abs{x}^2 + \abs{y}^2 + \abs{u}^2)
        \end{align*}
        for all $x, u \in \R^{d_1}$, $y,v \in \R^{d_2}$, and $s \in [0,1]$.
    \end{enumerate}
\end{condition}

Throughout this paper, we work with independent standard Brownian motions $W^{X,k_1}$ on $\R^{d_1}$ and $W^{Y,k_2}$ on $\R^{d_2}$.
With those conditions above, the well-posedness and boundedness of the particle systems are guaranteed in the same manner as~\cite[Theorem 3]{HuangQiuRiedl2024}.
\begin{proposition}
    Assume that Conditions~\ref{cd:cutoff} and~\ref{cd:efn} hold.
    Let $N_1, N_2 \in \N_+$, and parameters $\lambda_1, \lambda_2, \sigma_1, \sigma_2, \alpha, \beta > 0$.
    Also, let $\bar\nu^X_0 \in \prob(B_{d_1}(0, R_\cut))$ and $\bar\nu^Y_0 \in \prob(B_{d_2}(0, R_\cut))$.
    The system of SDEs~\eqref{eq:algorithm}
    admits a unique strong solution on the time interval $[0, \infty)$.
    In addition, the solution remains in the compact domain in the sense that
    \begin{equation*}
        \sup_{t \ge 0} \sup_{k_1 \in [N_1]} \abs{X^{N_1,k_1}_t} \le R_\cut \,, \qquad \sup_{t \ge 0} \sup_{k_2 \in [N_2]} \abs{Y^{N_2,k_2}_t} \le R_\cut \,, \qquad \text{a.s.}
    \end{equation*}
\end{proposition}

\subsection{Propagation of chaos}\label{s:poc}

When the numbers of particles  $N_1$ and $N_2$  approach to $+\infty$,~\cite{HuangQiuRiedl2024} informally derives the mean-field limit of the particle system~\eqref{eq:algorithm}. 
Specifically, the distribution of the particle system~\eqref{eq:algorithm} will converge to its mean-field limit $(\bar\nu^X_t, \bar\nu^Y_t)_{t \ge 0}$, which is the law of a system of SDEs 
\begin{subequations}
\label{eq:mean-field-limit}
    \begin{equation}
        d \bar X_t = -\lambda_1 (\bar X_t - \cM_\alpha(\bar\nu^X_t, \bar\nu^Y_t)) dt + \sigma_1 D_1(\bar X_t, \cM_\alpha(\bar\nu^X_t, \bar\nu^Y_t)) d\bar W^X_t \,,
    \end{equation}
    \begin{equation}
        d \bar Y_t = -\lambda_2 (\bar Y_t - \cM_{-\beta}(\bar\nu^Y_t, \bar\nu^X_t)) dt + \sigma_2 D_2(\bar Y_t, \cM_{-\beta}(\bar\nu^Y_t, \bar\nu^X_t)) d\bar W^Y_t \,,
    \end{equation}
    given the initial data 
    \begin{equation}
        \bar X_0 \sim \bar\nu^X_0 \,, \qquad \bar Y_0 \sim \bar\nu^Y_0 \,.
    \end{equation}
\end{subequations}
A similar well-posedness and boundedness argument holds for~\eqref{eq:mean-field-limit} in the same manner as~\cite[Theorem 6]{HuangQiuRiedl2024}.
\begin{proposition}
    Assume that Conditions~\ref{cd:cutoff} and~\ref{cd:efn} hold.
    Let $N_1, N_2 \in \N_+$.
    Suppose $\bar W^X$ on $\R^{d_1}$ and $\bar W^Y$ on $\R^{d_2}$ are independent Brownian motions.
    Let $\lambda_1, \lambda_2, \sigma_1, \sigma_2, \alpha, \beta > 0$, and $\bar\nu^X_0 \in \prob(B_{d_1}(0, R_\cut))$, $\bar\nu^Y_0 \in \prob(B_{d_2}(0, R_\cut))$.
    The system of SDEs~\eqref{eq:mean-field-limit} admits a unique strong solution on the time interval $[0, \infty)$.
    In addition, the solution remains in the compact domain in the sense that
    \begin{equation*}
        \sup_{t \ge 0} \abs{\bar X_t} \le R_\cut \,, \qquad \sup_{t \ge 0} \abs{\bar Y_t} \le R_\cut \,, \qquad \text{a.s.}
    \end{equation*}
\end{proposition}

Now that we have the well-posedness for all the stochastic systems above, we may work on the propagation of chaos.
The purpose of this work is to show the uniform-in-time proximity of the empirical distributions $(\nu_{\bX^{N_1}_t}, \nu_{\bY^{N_2}_t})_{t \ge 0}$ to its mean-field limit $(\bar\nu^X_t, \bar\nu^Y_t)_{t \ge 0}$ via the particle coupling method. 
More precisely, we create independent copies $\{\bar X^{k_1}\}_{k_1 \in \N_+}$ and $\{\bar Y^{k_2}\}_{k_2 \in \N_+}$ of $\bar X$ and $\bar Y$, respectively, such that
\begin{align*}
    d \bar X^{k_1}_t & = -\lambda_1 (\bar X^{k_1}_t - \cM_\alpha(\bar\nu^X_t, \bar\nu^Y_t)) dt + \sigma_1 D_1(\bar X^{k_1}_t, \cM_\alpha(\bar\nu^X_t, \bar\nu^Y_t)) dW^{X,k_1}_t \,,\\
    d \bar Y^{k_2}_t & = -\lambda_2 (\bar Y^{k_2}_t - \cM_{-\beta}(\bar\nu^Y_t, \bar\nu^X_t)) dt + \sigma_2 D_2(\bar Y^{k_2}_t, \cM_{-\beta}(\bar\nu^Y_t, \bar\nu^X_t)) dW^{Y,k_2}_t \,.
\end{align*}
In particular, the Brownian motions $W^{X,k_1}$'s and $W^{Y,k_2}$ are identical to those in~\eqref{eq:algorithm}.
Define the $L^2$-distance between these independent copies and the particle system~\eqref{eq:algorithm}:
\begin{equation*}
    \cD^X_{N_1}(t) \defeq \frac{1}{N_1} \sum_{j=1}^{N_1} \abs{X^{N_1,j}_t - \bar X^j_t}^2 \,, \qquad \cD^Y_{N_2}(t) \defeq \frac{1}{N_2} \sum_{j=1}^{N_2} \abs{Y^{N_2,j}_t - \bar Y^j_t}^2 \,.
\end{equation*}
The next theorem, which is the main result of this work, shows that the distances are well-controlled and converge to 0 as the numbers of particles increase.

Before stating the theorem, we need to define several constants: 
\begin{subequations}
for $p \ge 1$, $C_{\text{MZ},p}$ and $C_{\text{BDG},p}$ are the sharpest constant terms arising from the Marcinkiewicz–Zygmund (MZ) inequality and the Burkholder-Davis-Gundy (BDG) inequality, respectively;
for $q \ge 1$,
\begin{equation}
\begin{aligned}
    C_{\text{rate},1}(q) & = 2q\left(\lambda_1 - (2q-1)\sigma_1^2 (1+e^{2\alpha C_\cE(1+2R_\cut^2)}) \right) \,, \\
    C_{\text{rate},2}(q) & = 2q\left(\lambda_2 - (2q-1)\sigma_2^2 (1+e^{2\beta C_\cE(1+2R_\cut^2)}) \right) \,, 
\end{aligned}
\end{equation}
and those lead to
\begin{equation}
\begin{aligned}
    C_{\text{tail,1}}(4) & = 2^{12} C_{\text{MZ},8} + 2^{11} \sigma_1^4 C_{\text{BDG},4} (C_{\text{rate,1}}(4) - 4\kappa)^{-2}  (1 + e^{2\alpha C_\cE(1+2R_\cut^2)})^{\frac 12} \,, \\
    C_{\text{tail,2}}(4) & = 2^{12} C_{\text{MZ},8} + 2^{11} \sigma_2^4 C_{\text{BDG},4} (C_{\text{rate,2}}(4) - 4\kappa)^{-2} (1 + e^{2\beta C_\cE(1+2R_\cut^2)})^{\frac 12} \,, \\
    \bar C_{\text{tail,1}}(4) & = \frac{81}{16} C_{\text{tail,1}}(4) + (1+ \frac{\sigma_1^2 (1 + e^{\alpha C_\cE(1+2R_\cut^2)})^2}{C_{\text{rate,1}}(1)-\kappa})^4 \times  \\
    & \qquad \frac{2^{11} 3^4 \sigma_1^{8}}{(C_{\text{rate,1}}(4) - 4\kappa)^4} {C_{\text{MZ},8} e^{8\alpha (C_{\overline{\cE}} + C_{\underline{\cE}}) (1+R_\cut^2)}} (1+e^{2\alpha C_\cE(1+2R_\cut^2)})^4 \,,\\
    \bar C_{\text{tail,2}}(4) & = \frac{81}{16} C_{\text{tail,2}}(4) + (1+ \frac{\sigma_2^2 (1 + e^{\beta C_\cE(1+2R_\cut^2)})^2}{C_{\text{rate,2}}(1)-\kappa})^4 \times  \\
    & \qquad \frac{2^{11} 3^4 \sigma_2^{8} }{(C_{\text{rate,2}}(4) - 4\kappa)^4} {C_{\text{MZ},8} e^{8\beta (C_{\overline{\cE}} + C_{\underline{\cE}}) (1+R_\cut^2)}} (1+e^{2\beta C_\cE(1+2R_\cut^2)})^4 \,,
\end{aligned}
\end{equation}
with $\kappa = \frac{C_{\text{rate,1}}(4) \land C_{\text{rate,2}}(4)}{8}$;
lastly, we define $\zeta = \frac{\kappa}{2}$, and 
\begin{equation}
\begin{aligned}
    C_{\text{decay}} & = 2\bar\lambda + 72(\bar\lambda+6\bar\sigma^2) \gamma^2 e^{8\gamma C_\cE (1+R_\cut^2)} (1+R_\cut^2) + 6\bar\sigma^2 (1+R_\cut^2) (1+e^{2\gamma C_\cE(1+2R_\cut^2)}) \norm{\grad\phi}_\infty^2 \,, \\
    C_{\text{error}} & = (2\bar\lambda+6\bar\sigma^2) C_{\text{MZ,2}} e^{2\gamma (C_{\overline{\cE}} + C_{\underline{\cE}})(1+R_\cut^2)} (1 + e^{2\gamma C_\cE(1+2R_\cut^2)}) R_\cut^2   \\
    & \qquad + 6\bar\sigma^2 \sqrt{C_{\text{tail,1}}(4) \lor C_{\text{tail,2}}(4)} (2R_\cut)^8 (1 + e^{2\gamma C_\cE (1+2R_\cut^2)}) \norm{\grad\phi}_\infty^2  \\
    & \qquad + (2\bar\lambda+6\bar\sigma^2) C_{\text{MZ,4}} e^{4\gamma C_{\overline{\cE}}(1+R_\cut^2)} L_\cE^2 (2R_\cut)^4 (1+4R_\cut)^2  \\
    & \qquad + 18432 (\bar\lambda+6\bar\sigma^2) \gamma^2e^{8\gamma C_\cE(1+R_\cut^2)} R_\cut^8 \times \\ 
    & \qquad \qquad (\sqrt{C_{\text{tail,1}}(4) + \bar C_{\text{tail,1}}(4)} \lor \sqrt{C_{\text{tail,2}}(4) + \bar C_{\text{tail,2}}(4)})   
\end{aligned}
\end{equation}
with $\bar\lambda = \lambda_1 \lor \lambda_2$, $\bar\sigma = \sigma_1 \lor \sigma_2$, and $\gamma = \alpha \lor \beta$.
\end{subequations}

\begin{theorem}
\label{th:main}
    Assume that Conditions~\ref{cd:cutoff} and~\ref{cd:efn} hold.
    Suppose $\lambda_1 > 15\sigma_1^2 (1+e^{2\alpha (1+R_\cut^2)}) \lor 3\sigma_1^2 (1 + 4R_\cut^2 \norm{\grad\phi}_\infty^2)$ and $\lambda_2 > 15\sigma_2^2 (1+e^{2\beta (1+R_\cut^2)}) \lor 3\sigma_2^2 (1 + 4R_\cut^2 \norm{\grad\phi}_\infty^2)$.
    With the constants $\zeta, C_{\text{decay}}, C_{\text{error}}$ defined above, we have
    \begin{equation}
    \label{eq:main}
        \E[\cD^X_{N_1}(t) + \cD^Y_{N_2}(t)] \le \left( \E[\cD^X_{N_1}(0) + \cD^Y_{N_2}(0)] + C_{\text{error}}\zeta^{-1} (N_1^{-1} + N_2^{-1}) \right) e^{\frac{C_{\text{decay}}}{\zeta}}
    \end{equation}
    holds for all $t \ge 0$ and $N_1, N_2 \in \N_+$.
    When the initial data of the copies $\bar X^{k_1}$ and $\bar Y^{k_2}$ are given by exactly $\bar X^{k_1}_0 = X^{N_1,k_1}_0$ and $\bar Y^{k_2}_0 = Y^{N_2,k_2}_0$ for $k_1 \in [N_1]$ and $k_2 \in [K_2]$, the result~\eqref{eq:main} simplifies to
    \begin{equation*}
        \E[\cD^X_{N_1}(t) + \cD^Y_{N_2}(t)] \le C_{\text{error}} \zeta^{-1} e^{\frac{C_{\text{decay}}}{\zeta}} (N_1^{-1} + N_2^{-1}) \,.
    \end{equation*}
\end{theorem}
The proof will be presented in Section~\ref{s:prf-main}.
As a consequence of Theorem~\ref{th:main}, we obtain an estimate on the weak error between the particle system~\eqref{eq:algorithm} and the mean-field limit~\eqref{eq:mean-field-limit}, in the Wasserstein metric.
\begin{corollary}
    Let $N_1, N_2 \in \N_+$.
    Assume that the conditions of Theorem~\ref{th:main} hold.
    There exist some constants $C_{\text{main}}, C_d > 0$, independent of $N_1$ and $N_2$, such that
    \begin{equation*}
        \sup_{t \ge 0} \E [\cW_2^2(\nu_{\bX^{N_1}_t}, \bar\nu^X_t) + \cW_2^2(\nu_{\bY^{N_2}_t}, \bar\nu^Y_t)] \le C_{\text{main}} (N_1^{-1} + N_2^{-1}) + C_d R_\cut^2 (\delta_d(N_1) + \delta_d(N_2)) \,,
    \end{equation*}
    where $C_{\text{main}} = C_{\text{error}} \zeta^{-1} e^{\frac{C_{\text{decay}}}{\zeta}}$, $C_d$ depends on $d$ only, and 
    \begin{equation*}
        \delta_d(n) \defeq \begin{cases}
            n^{-\frac 12} & \text{if } d < 4 \,, \\
            n^{-\frac 12} \log n & \text{if } d = 4 \,, \\
            n^{-\frac 2d} & \text{if } d > 4 \,.
        \end{cases}
    \end{equation*}
\end{corollary}

\begin{proof}
    We construct $\{\bar X^{k_1}_t\}_{k_1 \in [N_1]}$ and $\{\bar Y^{k_2}_t\}_{k_2 \in [N_2]}$ with initial data $\bar X^{k_1}_0 = X^{N_1,k_1}_0$ and $\bar Y^{k_2}_0 = Y^{N_2,k_2}_0$.
    Notice that they are i.i.d.~samples of $\bar\nu^X_t$ and $\bar\nu^Y_t$, respectively, for every $t \ge 0$.
    With~\cite[Theorem 1]{FournierGuillin2015}, we have 
    \begin{equation*}
        \E[\cW_2^2(\nu_{\bar\bX^{N_1}_t}, \bar\nu^X_t) + \cW_2^2(\nu_{\bar\bY^{N_2}_t}, \bar\nu^Y_t)] \le C_d R_\cut^2 (\delta_d(N_1) + \delta_d(N_2)) \,,
    \end{equation*}
    where $R_\cut^2$ is a uniform-in-time upper bound for the second moments of all particles.
    Joining this with Theorem~\ref{th:main} and applying the triangular inequality leads to the conclusion.
\end{proof}
 
\begin{remark}
    We make a few comments on the previous results.
    \begin{enumerate}
        \item The above corollary gives the uniform-in-time weak propagation of chaos of the saddle point algorithm. 

        \item It is shown in~\cite[Theorem 11]{HuangQiuRiedl2024} that $(\bar\nu^X_t, \bar\nu^Y_t)$ converges to some limit point $(\tilde x, \tilde y)$ in distribution as $t \to \infty$. 
        Subject to certain structural properties of the objective function $\cE$ and appropriate choices of scalars $\alpha$ and $\beta$, the limit point is close to a saddle point $(x^\ast, y^\ast)$ such that
        \begin{equation*}
            \cE(x^\ast, y^\ast) = \min_{x \in \R^{d_1}} \max_{y \in \R^{d_2}} \cE(x,y) \,.
        \end{equation*}
        Theorem~\ref{th:main} tells that one can run the particle system~\eqref{eq:algorithm} to find an approximation of a saddle point by choosing some large $\alpha$ and $\beta$.
    \end{enumerate}
\end{remark}

\section{Strategies of the proof}

Classical propagation of chaos results hold up to a finite time horizon in the sense that $\E [\cD^X_{N_1}(t) + \cD^Y_{N_2}(t)] \le C_t (N_1^{-\delta}+N_2^{-\delta})$, where $C_t$ typically grows fast with $t$.
This work, however, yields a better result, thanks to the mean-reverting
dynamics of the consensus-based algorithm.
We will prove the following decay
\begin{align}
\label{eq:differential}
    \frac{d \E[\cD^X_{N_1}(t) + \cD^Y_{N_2}(t)]}{dt} \le C_{\text{decay}} e^{-\zeta t} \E[\cD^X_{N_1}(t) + \cD^Y_{N_2}(t)] + C_{\text{error}} e^{-\zeta t} (N_1^{-1} + N_2^{-1}) \,,
\end{align}
with the constants independent of $t$, which leads to the main result~\eqref{eq:main}.
The core of approaching inequality~\eqref{eq:differential} is the estimate of the generalized variance $\cV_{2q}$, which characterizes the clustering speed of the particle systems.

\subsection{Estimates on variance}\label{s:var}

Notice that the dynamics of~\eqref{eq:algorithm} are mean-reverting.
Recall the \emph{$2q$-th order generalized variance} of a probability measure defined by
\begin{equation*}
    \cV_{2q}(\mu) = \int \abs{x - M(\mu)}^{2q} \mu(dx) \,,
\end{equation*}
where $M$ is the mean operator $\mu \mapsto \int x \mu(dx)$.
Note that $M(\mu) = \cM_0(\mu, \_)$.
With $\lambda$'s sufficiently large, we also have the following convergence result.
It shows the clustering effect of the finite particle systems.
\begin{lemma}
\label{lm:var-decay-particles}
    Assume that Conditions~\ref{cd:cutoff} and~\ref{cd:efn} hold.
    Let $q \ge 1$.
    Suppose $\lambda_1 > (2q-1) \sigma_1^2 (1+e^{2\alpha C_\cE(1+2R_\cut^2)})$ and $\lambda_2 > (2q-1)\sigma_2^2(1+e^{2\beta C_\cE(1+2R_\cut^2)})$.
    Then the generalized variances admit the following decays:
    \begin{equation*}
        \E \cV_{2q}(\nu_{\bX^{N_1}_t}) \le e^{-C_{\text{rate},1}(q) t} \E \cV_{2q}(\nu_{\bX^{N_1}_0})  \,, \qquad  \E \cV_{2q}(\nu_{\bY^{N_2}_t}) \le e^{-C_{\text{rate},2}(q) t} \E \cV_{2q}(\nu_{\bY^{N_2}_0})  \,,
    \end{equation*}
    where
    \begin{align*}
        C_{\text{rate},1}(q) & = 2q\left(\lambda_1 - (2q-1)\sigma_1^2 (1+e^{2\alpha C_\cE(1+2R_\cut^2)}) \right) \,, \\C_{\text{rate},2}(q)  &= 2q\left(\lambda_2 - (2q-1)\sigma_2^2 (1+e^{2\beta C_\cE(1+2R_\cut^2)}) \right) \,.
    \end{align*}
\end{lemma}
\begin{remark}
    In Lemma~\ref{lm:var-decay-particles}, the lower bounds on $\lambda_i$, $i=1,2$, ensure that the drift terms dominate the diffusion terms in the dynamics, leading to the clustering of particles.
    Notice that the choices of $\lambda_i$ and $\sigma_i$ are independent of $N_1$ and $N_2$.
    This means the (generalized) variances converge at the same rates regardless of the size of the particle systems.
\end{remark}

Analogously, the cluster property also holds for the mean-field limits.
\begin{lemma}
\label{lm:var-decay-meanfield}
    Assume that Conditions~\ref{cd:cutoff} and~\ref{cd:efn} hold.
    Let $q \ge 1$.
    Suppose $\lambda_1 > (2q-1) \sigma_1^2 (1+e^{\alpha C_\cE(1+2R_\cut^2)})$ and $\lambda_2 > (2q-1)\sigma_2^2 (1+e^{\beta C_\cE(1+2R_\cut^2)}$.
    Then the generalized variances admit the following decays:
    \begin{equation*}
        \cV_{2q}(\bar\nu^X_t) \le e^{-C_{\text{rate},1}(q) t} \cV_{2q}(\bar\nu^X_0)  \,, \qquad  \cV_{2q}(\bar\nu^Y_t) \le e^{-C_{\text{rate},2}(q) t} \cV_{2q}(\bar\nu^Y_0)  \,,
    \end{equation*}
    with $C_{\text{rate},i}(q)$, $i=1,2$, the same as defined in Lemma~\ref{lm:var-decay-particles}.
\end{lemma}

We recall the i.i.d.~copies $\bar X^{k_1}$ and $\bar Y^{k_2}$ of $\bar X$ and $\bar Y$.
The above lemma also leads to a bound on the variance of the particle system $(\bar\bX^{N_1}, \bar\bY^{N_2})$.
\begin{corollary}
\label{cr:var-decay-copies}
    Let $N_1, N_2 \in \N_+$ and $q \ge 1$.
    For all $t \ge 0$,
    \begin{align*}
        \E \var{\nu_{\bar\bX^{N_1}_t}} \le  e^{-C_{\text{rate,1}}(1) t} \var{\bar\nu^X_0}  \,,\\
        \E \var{\nu_{\bar\bY^{N_2}_t}} \le  e^{-C_{\text{rate,2}}(1) t} \var{\bar\nu^Y_0} \,.
    \end{align*}
\end{corollary}
\begin{proof}
    We observe that, for i.i.d.~random variables $Z_1, \dots, Z_n$,
    \begin{equation*}
        \E \var{\nu_{\bZ}} = \E \left[ \frac 1n \sum_{j=1}^n \abs{Z_j - M(\nu_\bZ)}^2 \right] = \E \abs{Z_1 - M(\nu_\bZ)}^2 = (1-n^{-1}) \var{Z_1} \,.
    \end{equation*}
\end{proof}

\begin{remark}
    Notice that $\cV_{2q}$ is deterministic for the mean-field processes.
    The above lemma states that, given initial data $(\bar\nu^X_0, \bar\nu^Y_0)$, there exists a unique pair of points $(\tilde x, \tilde y) \in \R^{d_1+d_2}$ depending on the initial data such that $(\bar X_t, \bar Y_t) \to (\tilde x, \tilde y)$ as $t \to \infty$.
    Theorem 11 in~\cite{HuangQiuRiedl2024} shows that $(\tilde x, \tilde y)$ is close to the solution to~\eqref{eq:minimax}, subject to some constraints on the objective function $\cE$ and sufficiently large scalars $\alpha$ and $\beta$.
\end{remark}

The above lemmata present exponential decays of the (expected) variance.
We now provide the upper bounds on the tail probability of the variance.
These are the analogs of the concentration inequalities in~\cite[Lemmata 4.9 and 4.12]{GerberHoffmannKimVaes2025}.


\begin{lemma}
    \label{lm:tail-bound-particles}
    Assume that Conditions~\ref{cd:cutoff} and~\ref{cd:efn} hold.
    For $q \ge 2$, suppose $\lambda_1 > (2q-1) \sigma_1^2 (1+e^{2\alpha C_\cE(1+2R_\cut^2)})$ and $\lambda_2 > (2q-1)\sigma_2^2(1+e^{2\beta C_\cE(1+2R_\cut^2)})$.
    Then there exist some constants $C_{\text{tail,1}}(q), C_{\text{tail,2}}(q)$, and $\kappa < \frac{C_{\text{rate,1}}(q) \land C_{\text{rate,2}}(q)}{q}$ such that, for all $A > 0$ and $N_1, N_2 \in \N_+$,
    \begin{equation*}
        \P \left[ \sup_{t \ge 0} e^{\kappa t} \var{\nu_{\bX^{N_1}_t}} \ge \var{\bar\nu^X_0} + A \right] \le C_{\text{tail,1}}(q) A^{-q} N_1^{-\frac q2} \cV_{2q}(\bar \nu^X_0)\,,
    \end{equation*}
    and
    \begin{equation*}
        \P \left[ \sup_{t \ge 0} e^{\kappa t} \var{\nu_{\bY^{N_2}_t}} \ge \var{\bar\nu^Y_0} + A \right] \le C_{\text{tail,2}}(q) A^{-q} N_2^{-\frac q2} \cV_{2q}(\bar \nu^Y_0) \,.
    \end{equation*}
    Explicitly, we have
    \begin{align*}
        C_{\text{tail,1}}(q) & = 2^{3q} C_{\text{MZ},2q} + 2^{3q} \sigma_1^q C_{\text{BDG},q} (C_{\text{rate,1}}(q) - q\kappa)^{-\frac q2} (q-2)^{-\frac{q-2}{2}} (1 + e^{2\alpha C_\cE(1+2R_\cut^2)})^{\frac 12} \,, \\
        C_{\text{tail,2}}(q) & = 2^{3q} C_{\text{MZ},2q} + 2^{3q} \sigma_2^q C_{\text{BDG},q} (C_{\text{rate,2}}(q) - q\kappa)^{-\frac q2} (q-2)^{-\frac{q-2}{2}} (1 + e^{2\beta C_\cE(1+2R_\cut^2)})^{\frac 12} \,, 
    \end{align*}
    with convention $0^0 = 1$, where $C_{\text{MZ},2q}$ and $C_{\text{BDG},q}$ are the constant terms arising from the Marcinkiewicz–Zygmund (MZ) inequality and the Burkholder-Davis-Gundy (BDG) inequality, respectively.
\end{lemma}

For the uniformity of notation, we write
\begin{align*}
    \bar\bX^{N_1} & = (\bar X^{1}, \dots, \bar X^{N_1}) \in \R^{N_1 d_1} \\
    \bar\bY^{N_2} & = (\bar Y^{1}, \dots, \bar Y^{N_2}) \in \R^{N_2 d_2}
\end{align*}
for the independent copies of the mean-field limit processes.
An identical argument holds as well.

\begin{lemma}
    \label{lm:tail-bound-copies}
    Assume that Conditions~\ref{cd:cutoff} and~\ref{cd:efn} hold.
    For $q \ge 2$, suppose $\lambda_1 > (2q-1) \sigma_1^2 (1+e^{2\alpha C_\cE(1+2R_\cut^2)})$ and $\lambda_2 > (2q-1)\sigma_2^2(1+e^{2\beta C_\cE(1+2R_\cut^2)})$.
    Then there exist some constants $\bar C_{\text{tail,1}}(q), \bar C_{\text{tail,2}}(q)$, and $\kappa < \frac{C_{\text{rate,1}}(q) \land C_{\text{rate,2}}(q)}{q}$ such that, for all $A > 0$ and $N_1, N_2 \in \N_+$,
    \begin{equation*}
        \P \left[ \sup_{t \ge 0} e^{\kappa t} \var{\nu_{\bar\bX^{N_1}_t}} \ge \var{\bar\nu^X_0} + A \right] \le \bar C_{\text{tail,1}}(q) A^{-q} N_1^{-\frac q2} \cV_{2q}(\bar \nu^X_0)
    \end{equation*}
    and
    \begin{equation*}
        \P \left[ \sup_{t \ge 0} e^{\kappa t} \var{\nu_{\bar\bY^{N_2}_t}} \ge \var{\bar\nu^Y_0} + A \right] \le \bar C_{\text{tail,2}}(q) A^{-q} N_2^{-\frac q2} \cV_{2q}(\bar \nu^Y_0) \,.
    \end{equation*}
    Explicitly, we have
    \begin{align*}
        \bar C_{\text{tail,1}}(q) & = \frac{3^q}{2^q} C_{\text{tail,1}}(q) + (1+ \frac{\sigma_1^2 (1 + e^{\alpha C_\cE(1+2R_\cut^2)})^2}{C_{\text{rate,1}}(1)-\kappa})^q \times  \\
        & \qquad \frac{2^{3q-1} \sigma_1^{2q} (q-1)^q}{(C_{\text{rate,1}}(q) - q\kappa)^q} {C_{\text{MZ},2q} e^{2q\alpha (C_{\overline{\cE}} + C_{\underline{\cE}}) (1+R_\cut^2)}} (1+e^{2\alpha C_\cE(1+2R_\cut^2)})^q \,,\\
        \bar C_{\text{tail,2}}(q) & = \frac{3^q}{2^q} C_{\text{tail,2}}(q) + (1+ \frac{\sigma_2^2 (1 + e^{\beta C_\cE(1+2R_\cut^2)})^2}{C_{\text{rate,2}}(1)-\kappa})^q \times  \\
        & \qquad \frac{2^{3q-1} \sigma_2^{2q} (q-1)^q}{(C_{\text{rate,2}}(q) - q\kappa)^q} {C_{\text{MZ},2q} e^{2q\beta (C_{\overline{\cE}} + C_{\underline{\cE}}) (1+R_\cut^2)}} (1+e^{2\beta C_\cE(1+2R_\cut^2)})^q  \,.
    \end{align*}
\end{lemma}

The proofs of all the above lemmata are presented in Section~\ref{s:prf-var}.
With those lemmata, we are able to proceed to the proof of the main theorem.

\subsection{Proof of Theorem~\ref{th:main}}
\label{s:prf-main}

The goal of this section is to demonstrate~\eqref{eq:differential}.
Since the positions of the $X$- and $Y$-particles are symmetric, it suffices to work with either and replace the parameters to obtain the other.
We neglect the subscript in $D_1$ for simplicity.

Given $N_1, N_2 \in \N_+$, we first apply It\^o's formula to $\abs{X^{N_1,k}_t - \bar X^k_t}^2$:
\begin{align*}
    & d \abs{X^{N_1,k}_t - \bar X^k_t}^2 = -2\lambda_1 (X^{N_1,k}_t - \bar X^k_t)^\top (X^{N_1,k}_t - \bar X^k_t - \cM_\alpha(\nu_{\bX^{N_1}_t}, \nu_{\bY^{N_2}_t}) + \cM_\alpha(\bar \nu^X_t, \bar \nu^Y_t) ) dt\\
    & \quad + 2\sigma_1 (X^{N_1,k}_t - \bar X^k_t)^\top \left( D(X^{N_1,k}_t, \cM_\alpha(\nu_{\bX^{N_1}_t}, \nu_{\bY^{N_2}_t})) - D(\bar X^k_t, \cM_\alpha(\bar \nu^X_t, \bar \nu^Y_t)) \right) dW^{X,k}_t \\
    & \quad + \sigma_1^2 \tr \left[ \left( D(X^{N_1,k}_t, \cM_\alpha(\nu_{\bX^{N_1}_t}, \nu_{\bY^{N_2}_t})) - D(\bar X^k_t, \cM_\alpha(\bar \nu^X_t, \bar \nu^Y_t)) \right)^\top \right. \\
    & \qquad \left. \left( D(X^{N_1,k}_t, \cM_\alpha(\nu_{\bX^{N_1}_t}, \nu_{\bY^{N_2}_t})) - D(\bar X^k_t, \cM_\alpha(\bar \nu^X_t, \bar \nu^Y_t)) \right) \right] dt
\end{align*}
for each $k \in [N_1]$. Take average in $k$ and define
\begin{align*}
    \cA_1(t) & \defeq -2\lambda_1 N_1^{-1} \sum_{k=1}^{N_1} \abs{X^{N_1,k}_t - \bar X^k_t}^2 dt + 2\lambda_1 \abs{M(\nu_{\bX^{N_1}_t}) - M(\nu_{\bar \bX^{N_1}_t})}^2   \\
    & \quad  - 2\lambda_1 M(\nu_{\bX^{N_1}_t} - \nu_{\bar \bX^{N_1}_t})^\top \left(M(\nu_{\bX^{N_1}_t}) - \cM_\alpha(\nu_{\bX^{N_1}_t}, \nu_{\bY^{N_2}_t}) - M(\nu_{\bar \bX^{N_1}_t}) + \cM_\alpha(\nu_{\bar \bX^{N_1}_t}, \nu_{\bar \bY^{N_2}_t}) \right) \,, \\
    \cA_2(t) & \defeq  2\lambda_1 M(\nu_{\bX^{N_1}_t} - \nu_{\bar \bX^{N_1}_t})^\top (\cM_\alpha(\nu_{\bar \bX^{N_1}_t}, \nu_{\bar \bY^{N_2}_t}) - \cM_\alpha(\bar \nu^X_t, \bar \nu^Y_t)) \,, \\
    \cA_3(t) & \defeq \frac{\sigma_1^2}{N_1} \sum_{k=1}^{N_1} \tr \left[ \left( D(X^{N_1,k}_t, \cM_\alpha(\nu_{\bX^{N_1}_t}, \nu_{\bY^{N_2}_t})) - D(\bar X^k_t, \cM_\alpha(\bar \nu^X_t, \bar \nu^Y_t)) \right)^\top \right. \\
    & \qquad \left. \left( D(X^{N_1,k}_t, \cM_\alpha(\nu_{\bX^{N_1}_t}, \nu_{\bY^{N_2}_t})) - D(\bar X^k_t, \cM_\alpha(\bar \nu^X_t, \bar \nu^Y_t)) \right) \right]
\end{align*}
Then
\begin{equation*}
    d \cD^X_{N_1}(t) = (\cA_1(t) + \cA_2(t) + \cA_3(t)) dt + \text{(martingale terms)}\,.
\end{equation*}
We will bound them consecutively.

\step[Bounding $\a_1$]
Applying the strategy in the proof of Theorem 2.1, \cite{GerberHoffmannKimVaes2025}, we have
\begin{align*}
    \E \cA_1(t) \le -\frac{2\lambda_1}{N_1} \sum_{k=1}^{N_1} \E \abs{ (X^{N_1,k}_t - M(\nu_{\bX^{N_1}_t})) - (\bar X^k_t - M(\nu_{\bar \bX^{N_1}_t}) }^2 + 2\lambda_1 \left( \cB(t) \E \cD^X_{N_1}(t) \right)^{\frac 12} \,,
\end{align*}
where $\cB(t) := \E \abs{ (M(\nu_{\bX^{N_1}_t}) - \cM_\alpha(\nu_{\bX^{N_1}_t}, \nu_{\bY^{N_2}_t})) - (M(\nu_{\bar \bX^{N_1}_t}) - \cM_\alpha(\nu_{\bar \bX^{N_1}_t}, \nu_{\bar \bY^{N_2}_t})) }^2$. By Lemma~\ref{l:decay-b}, we have
\begin{equation}
\label{eq:upper-bound-bx}
    \cB(t) \le 2C_M^2 \E \left[ \left( \var{\nu_{\bX^{N_1}_t}} + \var{\nu_{\bar\bX^{N_1}_t}} \right) \left( \cW_2^2(\nu_{\bX^{N_1}_t}, \nu_{\bar\bX^{N_1}_t}) + \cW_2^2(\nu_{\bY^{N_2}_t}, \nu_{\bar\bY^{N_2}_t}) \right) \right]
\end{equation}
with $C_M = 3\alpha e^{4\alpha C_\cE (1+R_\cut^2)}$.

Let $\kappa = \frac{C_{\text{rate,1}}(4) \land C_{\text{rate,2}}(4)}{8}$, and define the measurable event
\begin{equation*}
    O_\kappa = \left\{ \sup_{t \ge 0} e^{\kappa t} \var{\nu_{\bX^{N_1}_t}} \ge \var{\bar\nu^X_0} + 1 \right\} \cup \left\{ \sup_{t \ge 0} e^{\kappa t} \var{\nu_{\bar\bX^{N_1}_t}} \ge \var{\bar\nu^X_0} + 1 \right\} \,.
\end{equation*}
This allows us to split the right-hand side of~\eqref{eq:upper-bound-bx} by
\begin{equation}\label{eqn:VarW2_bound}
\begin{aligned}
    & \E \left[ \left( \var{\nu_{\bX^{N_1}_t}} + \var{\nu_{\bar\bX^{N_1}_t}} \right) \left( \cW_2^2(\nu_{\bX^{N_1}_t}, \nu_{\bar\bX^{N_1}_t}) + \cW_2^2(\nu_{\bY^{N_2}_t}, \nu_{\bar\bY^{N_2}_t}) \right) \right]  \\
    & = \E \left[ \one{O_\kappa} \left( \var{\nu_{\bX^{N_1}_t}} + \var{\nu_{\bar\bX^{N_1}_t}} \right) \left( \cW_2^2(\nu_{\bX^{N_1}_t}, \nu_{\bar\bX^{N_1}_t}) + \cW_2^2(\nu_{\bY^{N_2}_t}, \nu_{\bar\bY^{N_2}_t}) \right) \right] \\
    & \quad + \E \left[ \one{O_\kappa^\complement} \left( \var{\nu_{\bX^{N_1}_t}} + \var{\nu_{\bar\bX^{N_1}_t}} \right) \left( \cW_2^2(\nu_{\bX^{N_1}_t}, \nu_{\bar\bX^{N_1}_t}) + \cW_2^2(\nu_{\bY^{N_2}_t}, \nu_{\bar\bY^{N_2}_t}) \right) \right] \\
    & \le \P(O_\kappa)^{\frac 12} \sqrt{\E \left[ \left( \var{\nu_{\bX^{N_1}_t}} + \var{\nu_{\bar\bX^{N_1}_t}} \right)^2 \left( \cW_2^2(\nu_{\bX^{N_1}_t}, \nu_{\bar\bX^{N_1}_t}) + \cW_2^2(\nu_{\bY^{N_2}_t}, \nu_{\bar\bY^{N_2}_t}) \right)^2 \right]} \\
    & \quad + 2 e^{-\kappa t} (1 + \var{\bar\nu^X_0}) \E \left[ \cW_2^2(\nu_{\bX^{N_1}_t}, \nu_{\bar\bX^{N_1}_t}) + \cW_2^2(\nu_{\bY^{N_2}_t}, \nu_{\bar\bY^{N_2}_t}) \right]
\end{aligned}
\end{equation}
The last line above is simply bounded by
\begin{equation*}
    2 e^{-\kappa t} (1 + R_\cut^2) \E [\cD^X_{N_1}(t) + \cD^Y_{N_2}(t)] \,.
\end{equation*}
To bound the second-to-last line, we apply Lemma~\ref{lm:var-decay-particles} and Corollary~\ref{cr:var-decay-copies}, as well as Jensen's inequality, to obtain
\begin{align*}
    & \E \left[ \left( \var{\nu_{\bX^{N_1}_t}} + \var{\nu_{\bar\bX^{N_1}_t}} \right)^2 \left( \cW_2^2(\nu_{\bX^{N_1}_t}, \nu_{\bar\bX^{N_1}_t}) + \cW_2^2(\nu_{\bY^{N_2}_t}, \nu_{\bar\bY^{N_2}_t}) \right)^2 \right] \\
    & \le 4 \E \left[ \left( \cV_4(\nu_{\bX^{N_1}_t}) + \cV_4(\nu_{\bar\bX^{N_1}_t}) \right) \left( \cD^X_{N_1}(t)^2 + \cD^Y_{N_2}(t)^2 \right) \right] \\
    & \le 4 \sqrt{ \E \left( \cV_4(\nu_{\bX^{N_1}_t}) + \cV_4(\nu_{\bar\bX^{N_1}_t}) \right)^2 } \sqrt{ \E \left( \cD^X_{N_1}(t)^2 + \cD^Y_{N_2}(t)^2 \right)^2 } \\
    & \le 8 \sqrt{ \E \left( \cV_8(\nu_{\bX^{N_1}_t}) + \cV_8(\nu_{\bar\bX^{N_1}_t}) \right) } \sqrt{ \E \left( \frac{1}{N_1} \sum_{k_1=1}^{N_1} \abs{X^{N_1,k_1}_t - \bar X^{k_1}_t}^8 + \frac{1}{N_2} \sum_{k_2=1}^{N_2} \abs{Y^{N_2,k_2}_t - \bar Y^{k_2}_t}^8 \right) } \\
    & \le 8 \sqrt{e^{-C_{\text{rate,1}}(4) t} \E(\cV_8(\nu_{\bX^{N_1}_0}) + \cV_8(\bar\nu^X_0)) } \sqrt{2(2R_\cut)^8} \\
    & \le 2^{12} R_\cut^8 e^{-\frac{C_{\text{rate,1}}(4) t}{2}}  \,.
\end{align*}
Now, applying Lemmata~\ref{lm:tail-bound-particles} and~\ref{lm:tail-bound-copies} with $q = 4$ and $A = 1$, we have
\begin{equation*}
    \P(O_\kappa) \le (C_{\text{tail,1}}(4) + \bar C_{\text{tail,1}}(4)) N_1^{-2} \cV_8(\bar\nu^X_0) \,,
\end{equation*}
so that the second-to-last line of \eqref{eqn:VarW2_bound} is bounded by
\begin{equation*}
    2^{10} R_\cut^8 \sqrt{C_{\text{tail,1}}(4) + \bar C_{\text{tail,1}}(4)} N_1^{-1} e^{-\frac{C_{\text{rate,1}}(4) t}{4}} \,.
\end{equation*}

Thus, we have
\begin{equation*}
    \cB(t) \le C_{\cB} N_1^{-1} e^{-\frac{C_{\text{rate,1}}(4) t}{4}} + 4 C_M^2 e^{-\kappa t} (1 + R_\cut^2) \E[\cD^X_{N_1}(t) + \cD^Y_{N_2}(t)] \,,
\end{equation*}
where $C_M = 3\alpha e^{4\alpha C_\cE(1+R_\cut^2)}$ and
\begin{equation*}
    C_{\cB} = 2C_M^2 2^{10} R_\cut^8 \sqrt{C_{\text{tail,1}}(4) + \bar C_{\text{tail,1}}(4)}  \,.
\end{equation*}
This leads to
\begin{align*}
    \E\cA_1(t) & \le -\frac{2\lambda_1}{N_1} \sum_{k=1}^{N_1} \E \abs{ (X^{N_1,k}_t - M(\nu_{\bX^{N_1}_t})) - (\bar X^k_t - M(\nu_{\bar \bX^{N_1}_t}) }^2 + \lambda_1 e^{\zeta t} \cB(t) + \lambda_1 e^{-\zeta t} \E \cD^X_{N_1}(t) \\
    & \le -\frac{2\lambda_1}{N_1} \sum_{k=1}^{N_1} \E \abs{ (X^{N_1,k}_t - M(\nu_{\bX^{N_1}_t})) - (\bar X^k_t - M(\nu_{\bar \bX^{N_1}_t}) }^2 \\
    & \quad + \lambda_1 C_{\cB} N_1^{-1} e^{-(\frac{C_{\text{rate,1}}(4)}{4} - \zeta)t} + 4\lambda_1 C_M^2 e^{-(\kappa-\zeta) t} (1 + R_\cut^2) \E[\cD^X_{N_1}(t) + \cD^Y_{N_2}(t)] \\
    & \quad + \lambda_1 e^{-\zeta t} \E \cD^X_{N_1}(t)
\end{align*}
for some $\zeta > 0$ to be chosen later.

\step[Bounding $\a_2$]
By Hölder's inequality and Jensen's inequality, we have
\begin{equation*}
    \E \cA_2(t) \le 2\lambda_1 \left( \E \cD^X_{N_1}(t) \right)^{\frac 12} \left( \E \abs{\cM_\alpha(\nu_{\bar \bX^{N_1}_t}, \nu_{\bar \bY^{N_2}_t}) - \cM_\alpha(\bar \nu^X_t, \bar \nu^Y_t)}^2 \right)^{\frac 12} \,.
\end{equation*}
Recall that
\begin{equation*}
    -C_{\underline{\cE}}(1+\abs{y}^2) \le \underline{\cE}(y) \le \cE(x,y) \le \overline{\cE}(x) \le C_{\overline{\cE}}(1+\abs{x}^2)
\end{equation*}
for all $(x,y) \in \R^{d_1+d_2}$.
The second multiplicative component admits the bound
\begin{align*}
    & \E \abs{\cM_\alpha(\nu_{\bar \bX^{N_1}_t}, \nu_{\bar \bY^{N_2}_t}) - \cM_\alpha(\bar \nu^X_t, \bar \nu^Y_t)}^2 \\
    & \le e^{2\alpha C_{\overline{\cE}}(1+R_\cut^2)} \E \abs{\frac{1}{N_1} \sum_{k=1}^{N_1} (\bar X^k_t - \cM_\alpha(\bar\nu^X_t, \bar\nu^Y_t)) e^{-\alpha \cE(\bar X^k_t, M(\nu_{\bar \bY^{N_2}_t}))} }^2 \\
    & \le 2 e^{2\alpha C_{\overline{\cE}}(1+R_\cut^2)} \E \abs{\frac{1}{N_1} \sum_{k=1}^{N_1} (\bar X^k_t - \cM_\alpha(\bar\nu^X_t, \bar\nu^Y_t)) e^{-\alpha \cE(\bar X^k_t, M(\bar\nu^Y_t))} }^2 \\
    & \quad + 2e^{2\alpha C_{\overline{\cE}}(1+R_\cut^2)} \E \abs{\frac{1}{N_1} \sum_{k=1}^{N_1} (\bar X^k_t - \cM_\alpha(\bar\nu^X_t, \bar\nu^Y_t)) (e^{-\alpha \cE(\bar X^k_t, M(\nu_{\bar \bY^{N_2}_t}))} - e^{-\alpha \cE(\bar X^k_t, M(\bar\nu^Y_t))}) }^2 \,.
\end{align*}
Recall that 
\begin{equation*}
    \cM_\alpha(\bar\nu^X_t, \bar\nu^Y_t) = \frac{\E[\bar X_t e^{-\alpha \cE(\bar X^k_t, M(\bar\nu^Y_t))}]}{\E e^{-\alpha \cE(\bar X_t, M(\bar\nu^Y_t))}} \,,
\end{equation*}
so that the variables $(\bar X^k_t - \cM_\alpha(\bar\nu^X_t, \bar\nu^Y_t)) e^{-\alpha \cE(\bar X^k_t, M(\bar\nu^Y_t))}$, $k \in [N_1]$, are i.i.d. and have mean 0.
By the Marcinkiewicz–Zygmund inequality, there exists some constant $C_{\text{MZ,2}}$ such that 
\begin{align*}
    & \E \abs{\frac{1}{N_1} \sum_{k=1}^{N_1} (\bar X^k_t - \cM_\alpha(\bar\nu^X_t, \bar\nu^Y_t)) e^{-\alpha \cE(\bar X^k_t, M(\bar\nu^Y_t))} }^2 \\
    & \le \frac{C_{\text{MZ,2}}}{N_1} \E \abs{(\bar X^1_t - \cM_\alpha(\bar\nu^X_t, \bar\nu^Y_t)) e^{-\alpha \cE(\bar X^1_t, M(\bar\nu^Y_t))}}^2 \\
    & \le \frac{C_{\text{MZ,2}} e^{2\alpha C_{\underline{\cE}}(1+R_\cut^2)}}{N_1} \E \abs{\bar X^1_t - \cM_\alpha(\bar\nu^X_t, \bar\nu^Y_t)}^2 \\
    & \le \frac{C_{\text{MZ,2}} e^{2\alpha C_{\underline{\cE}}(1+R_\cut^2)}}{N_1} (1 + e^{2\alpha C_\cE(1+2R_\cut^2)}) \var{\bar\nu^X_t} \,,
\end{align*}
where the last step applies Lemma~\ref{l:mean-consensus}.
In addition, observe that
\begin{align*}
    & \abs{e^{-\alpha \cE(\bar X^k_t, M(\nu_{\bar \bY^{N_2}_t}))} - e^{-\alpha \cE(\bar X^k_t, M(\bar\nu^Y_t))}} \\
    & \le e^{\alpha C_{\overline{\cE}}(1+R_\cut^2)} \abs{\cE(\bar X^k_t, M(\nu_{\bar \bY^{N_2}_t})) - \cE(\bar X^k_t, M(\bar\nu^Y_t))} \\
    & \le e^{\alpha C_{\overline{\cE}}(1+R_\cut^2)} L_{\cE}(1 + 2\abs{\bar X^k_t} + \abs{M(\nu_{\bar \bY^{N_2}_t})} + \abs{M(\bar\nu^Y_t)}) \abs{M(\nu_{\bar \bY^{N_2}_t}) - M(\bar\nu^Y_t)} \,,
\end{align*}
where $L_\cE$ is the locally Lipschitz constant of $\cE$.
Then
\begin{align*}
    & \E \abs{\frac{1}{N_1} \sum_{k=1}^{N_1} (\bar X^k_t - \cM_\alpha(\bar\nu^X_t, \bar\nu^Y_t)) (e^{-\alpha \cE(\bar X^k_t, M(\nu_{\bar \bY^{N_2}_t}))} - e^{-\alpha \cE(\bar X^k_t, M(\bar\nu^Y_t))}) }^2 \\
    & \le \frac{e^{2\alpha C_{\overline{\cE}}(1+R_\cut^2)} L_\cE^2}{N_1} \sum_{k=1}^{N_1} \left( \E\abs{\bar X^k_t - \cM_\alpha(\bar\nu^X_t, \bar\nu^Y_t)}^8 \right)^{\frac 14} \times \\
    & \qquad \quad \left( \E(1 + 2\abs{\bar X^k_t} + \abs{M(\nu_{\bar \bY^{N_2}_t})} + \abs{M(\bar\nu^Y_t)})^8 \right)^{\frac 14} \left( \E\abs{M(\nu_{\bar \bY^{N_2}_t}) - M(\bar\nu^Y_t)}^4 \right)^{\frac 12} \,.
\end{align*}
The first two parentheses terms can be bounded by $O(R_\cut^2)$.
The last term expands as
\begin{equation*}
    \E\abs{M(\nu_{\bar \bY^{N_2}_t}) - M(\bar\nu^Y_t)}^4 = \E \abs{\frac{1}{N_2} \sum_{j=1}^{N_2} \bar Y^j_t - \E \bar Y^j_t}^4 \,,
\end{equation*}
where $\bar Y^j$'s are i.i.d., so by the Marcinkiewicz-Zygmund inequality,
\begin{equation*}
    \E\abs{M(\nu_{\bar \bY^{N_2}_t}) - M(\bar\nu^Y_t)}^4 \le C_{\text{MZ,4}} N_2^{-2} \cV_4(\bar \nu^Y_t) \,.
\end{equation*}
Recall that
\begin{equation*}
    \var{\bar\nu^X_t} \le e^{-C_{\text{rate,1}}(1) t} R_\cut^2 \,, \qquad \cV_4(\bar\nu^Y_t) \le e^{-C_{\text{rate,2}}(2) t} (2R_\cut)^4 \,.
\end{equation*}
Let $\zeta < \min\{C_{\text{rate,1}}(1), C_{\text{rate,2}}(2)/2\}$, by AM-GM inequality, we obtain
\begin{align*}
    \E \cA_2(t) & \le \lambda_1 e^{-\zeta t} \E \cD^X_{N_1}(t) \\
    & \quad + 2\lambda_1 C_{\text{MZ,2}} e^{2\alpha (C_{\overline{\cE}} + C_{\underline{\cE}})(1+R_\cut^2)} (1 + e^{2\alpha C_\cE(1+2R_\cut^2)}) R_\cut^2 N_1^{-1} e^{-(C_{\text{rate,1}}(1)-\zeta)t}  \\
    & \quad + 2\lambda_1 C_{\text{MZ,4}} e^{4\alpha C_{\overline{\cE}}(1+R_\cut^2)} L_\cE^2 (2R_\cut)^4 (1+4R_\cut)^2 N_2^{-1} e^{-(\frac{C_{\text{rate,2}}(2)}{2} - \zeta) t}  \,.
\end{align*}

\step[Bounding $\a_3$]
Notice that
\begin{align*}
    & \tr \left[ \left( D(X^{N_1,k}_t , \cM_\alpha(\nu_{\bX^{N_1}_t}, \nu_{\bY^{N_2}_t})) - D(\bar X^k_t , \cM_\alpha(\bar \nu^X_t, \bar \nu^Y_t)) \right)^\top \right. \\
    & \qquad \left. \left( D(X^{N_1,k}_t , \cM_\alpha(\nu_{\bX^{N_1}_t}, \nu_{\bY^{N_2}_t})) - D(\bar X^k_t , \cM_\alpha(\bar \nu^X_t, \bar \nu^Y_t)) \right) \right] \\
    & = \abs{(X^{N_1,k}_t - \cM_\alpha(\nu_{\bX^{N_1}_t}, \nu_{\bY^{N_2}_t})) \phi(X^{N_1,k}_t) - (\bar X^k_t - \cM_\alpha(\bar \nu^X_t, \bar \nu^Y_t)) \phi(\bar X^k_t) }^2 \\
    & \le 3 \abs{(X^{N_1,k}_t - M(\nu_{\bX^{N_1}_t})) \phi(X^{N_1,k}_t) - (\bar X^k_t - M(\nu_{\bar\bX^{N_1}_t})) \phi(\bar X^k_t) }^2 \\
    & \qquad + 3 \abs{(M(\nu_{\bX^{N_1}_t}) - \cM_\alpha(\nu_{\bX^{N_1}_t}, \nu_{\bY^{N_2}_t})) \phi(X^{N_1,k}_t) - (M(\nu_{\bar\bX^{N_1}_t}) - \cM_\alpha(\nu_{\bar\bX^{N_1}_t}, \nu_{\bar\bY^{N_2}_t})) \phi(\bar X^k_t) }^2 \\
    & \qquad + 3 \abs{\cM_\alpha(\nu_{\bar\bX^{N_1}_t}, \nu_{\bar\bY^{N_2}_t}) - \cM_\alpha(\bar \nu^X_t, \bar \nu^Y_t)}^2 \,.
\end{align*}
Then
\begin{align*}
    \cA_3(t) & \le \frac{3\sigma_1^2}{N_1} \sum_{k=1}^{N_1} \abs{(X^{N_1,k}_t - M(\nu_{\bX^{N_1}_t})) \phi(X^{N_1,k}_t) - (\bar X^k_t - M(\nu_{\bar\bX^{N_1}_t})) \phi(\bar X^k_t) }^2 \\
    & \quad + \frac{3\sigma_1^2}{N_1} \sum_{k=1}^{N_1} \abs{(M(\nu_{\bX^{N_1}_t}) - \cM_\alpha(\nu_{\bX^{N_1}_t}, \nu_{\bY^{N_2}_t})) \phi(X^{N_1,k}_t) - (M(\nu_{\bar\bX^{N_1}_t}) - \cM_\alpha(\nu_{\bar\bX^{N_1}_t}, \nu_{\bar\bY^{N_2}_t})) \phi(\bar X^k_t) }^2 \\
    & \qquad + 3 \sigma_1^2 \abs{\cM_\alpha(\nu_{\bar\bX^{N_1}_t}, \nu_{\bar\bY^{N_2}_t}) - \cM_\alpha(\bar \nu^X_t, \bar \nu^Y_t)}^2 \\
    & =: 3\sigma_1^2 (\cA_{3,1}(t) + \cA_{3,2}(t) + \cA_{3,3}(t)) \,.
\end{align*}

Since $\phi \in [0,1]$, we have
\begin{align*}
    \cA_{3,1}(t) & \le \frac{2}{N_1} \sum_{k=1}^{N_1} \abs{(X^{N_1,k}_t - M(\nu_{\bX^{N_1}_t})) - (\bar X^k_t - M(\nu_{\bar\bX^{N_1}_t})) }^2 \\
    & \qquad + \frac{2}{N_1} \sum_{k=1}^{N_1} \abs{(X^{N_1,k}_t - M(\nu_{\bX^{N_1}_t})) (\phi(X^{N_1,k}_t) - \phi(\bar X^k_t))}^2 \,.
\end{align*}
The first term merges into the negative term in $\cA_1(t)$ when $\lambda_1$ is sufficiently large compared to $\sigma_1^2$.
For the second term, we have
\begin{align*}
    & \frac{1}{N_1} \sum_{k=1}^{N_1} \abs{(X^{N_1,k}_t - M(\nu_{\bX^{N_1}_t}))(\phi(X^{N_1,k}_t) - \phi(\bar X^k_t))}^2\\
    & \le \frac{\norm{\grad\phi}_\infty^2}{N_1} \sum_{k=1}^{N_1}  \abs{(X^{N_1,k}_t - M(\nu_{\bX^{N_1}_t}))(X^{N_1,k}_t - \bar X^k_t)}^2\\
    &\le  \frac{2\norm{\grad\phi}_\infty^2}{N_1} \sum_{k=1}^{N_1} \abs{(X^{N_1,k}_t - M(\nu_{\bX^{N_1}_t}))((X^{N_1,k}_t - M(\nu_{\bX^{N_1}_t})) - (\bar X^k_t - M(\nu_{\bar\bX^{N_1}_t})))}^2 \\
    & \qquad +\frac{2\norm{\grad\phi}_\infty^2}{N_1} \sum_{k=1}^{N_1} \abs{(X^{N_1,k}_t - M(\nu_{\bX^{N_1}_t})) (M(\nu_{\bX^{N_1}_t}) - M(\nu_{\bar\bX^{N_1}_t})) }^2 \\
    & \le \frac{2(2R_{\cut})^2 \norm{\grad\phi}_\infty^2}{N_1} \sum_{k=1}^{N_1}  \abs{(X^{N_1,k}_t - M(\nu_{\bX^{N_1}_t})) - (\bar X^k_t - M(\nu_{\bar\bX^{N_1}_t}))}^2 \\
    & \qquad + 2\norm{\grad\phi}_\infty^2  \left[ \var{\nu_{\bX^{N_1}_t}} \abs{M(\nu_{\bX^{N_1}_t}) - M(\nu_{\bar\bX^{N_1}_t})}^2 \right] \\
    & \le \frac{2(2R_{\cut})^2 \norm{\grad\phi}_\infty^2}{N_1} \sum_{k=1}^{N_1}  \abs{(X^{N_1,k}_t - M(\nu_{\bX^{N_1}_t})) - (\bar X^k_t - M(\nu_{\bar\bX^{N_1}_t}))}^2 \\
    & \qquad + 2\norm{\grad\phi}_\infty^2 \left[ \var{\nu_{\bX^{N_1}_t}} \cD^X_{N_1}(t) \right] \,.
\end{align*}
In total we have
\begin{align*}
    \cA_{3,1}(t) & \le (2 + 8R_\cut^2 \norm{\grad\phi}_\infty^2) N_1^{-1} \sum_{k=1}^{N_1} \abs{(X^{N_1,k}_t - M(\nu_{\bX^{N_1}_t})) - (\bar X^k_t - M(\nu_{\bar\bX^{N_1}_t}))}^2 \\
    & \qquad + 2 \norm{\grad\phi}_\infty^2 \left[ \var{\nu_{\bX^{N_1}_t}} \cD^X_{N_1}(t) \right] \,.
\end{align*}

Analogously, for $\cA_{3,2}$ we have
\begin{align*}
    \E\cA_{3,2}(t) & \le 2 \E\abs{(M(\nu_{\bX^{N_1}_t}) - \cM_\alpha(\nu_{\bX^{N_1}_t}, \nu_{\bY^{N_2}_t})) - (M(\nu_{\bar\bX^{N_1}_t}) - \cM_\alpha(\nu_{\bar\bX^{N_1}_t}, \nu_{\bar\bY^{N_2}_t})) }^2 \\
    & \qquad + \frac{2}{N_1} \sum_{k=1}^{N_1} \E\abs{(M(\nu_{\bX^{N_1}_t}) - \cM_\alpha(\nu_{\bX^{N_1}_t}, \nu_{\bY^{N_2}_t})) (\phi(X^{N_1,k}_t) - \phi(\bar X^k_t)) }^2 \\
    & \le 2 \cB(t) + 2 \E\left[ \abs{M(\nu_{\bX^{N_1}_t}) - \cM_\alpha(\nu_{\bX^{N_1}_t}, \nu_{\bY^{N_2}_t})}^2 \norm{\grad\phi}_\infty^2 \cD^X_{N_1}(t) \right] \\
    & \le 2 \cB(t) + 2 e^{2\alpha C_\cE(1+2R_\cut^2)} \norm{\grad\phi}_\infty^2 \E\left[ \var{\nu_{\bX^{N_1}_t}} \cD^X_{N_1}(t) \right] \,.
\end{align*}
We already had an estimate for $\cB(t)$, and for the other term we apply Lemma~\ref{lm:tail-bound-particles} with $q=4$ and the same $\kappa$ as in Step 1 to obtain
\begin{equation*}
    \E\left[ \var{\nu_{\bX^{N_1}_t}} \cD^X_{N_1}(t) \right] \le \sqrt{C_{\text{tail,1}}(4)} (2R_\cut)^8 N_1^{-1} e^{-\frac{C_{\text{rate,1}}(2)}{2} t} + (1+R_\cut)^2 e^{-\kappa t} \E \cD^X_{N_1}(t) \,.
\end{equation*}

Finally, the bound for $\cA_{3,3}$ is given as in that of $\cA_2$.
Then we have
\begin{align*}
    \E \cA_3(t) & \le 3\sigma_1^2(2 + 8R_\cut^2 \norm{\grad\phi}_\infty^2) N_1^{-1} \sum_{k=1}^{N_1} \E \abs{(X^{N_1,k}_t - M(\nu_{\bX^{N_1}_t})) - (\bar X^k_t - M(\nu_{\bar\bX^{N_1}_t}))}^2 \\
    & \quad + 6\sigma_1^2 \cB(t) + 6\sigma_1^2 (1 + e^{2\alpha C_\cE (1+2R_\cut^2)}) \norm{\grad\phi}_\infty^2 \E \left[\var{\nu_{\bX^{N_1}_t}} \cD^X_{N_1}(t) \right] \\
    & \quad + 3 \sigma_1^2 \E\abs{\cM_\alpha(\nu_{\bar\bX^{N_1}_t}, \nu_{\bar\bY^{N_2}_t}) - \cM_\alpha(\bar \nu^X_t, \bar \nu^Y_t)}^2 \\ 
    & \le 3\sigma_1^2(2 + 8R_\cut^2 \norm{\grad\phi}_\infty^2) N_1^{-1} \sum_{k=1}^{N_1} \E \abs{(X^{N_1,k}_t - M(\nu_{\bX^{N_1}_t})) - (\bar X^k_t - M(\nu_{\bar\bX^{N_1}_t}))}^2 \\
    & \quad + 6\sigma_1^2 C_{\cB} N_1^{-1} e^{-\frac{C_{\text{rate,1}}(4) t}{4}} + 24 C_M^2 \sigma_1^2 e^{-\kappa t} (1 + R_\cut^2) \E[\cD^X_{N_1}(t) + \cD^Y_{N_2}(t)] \\
    & \quad + 6\sigma_1^2 \sqrt{C_{\text{tail,1}}(4)} (2R_\cut)^8 (1 + e^{2\alpha C_\cE (1+2R_\cut^2)}) \norm{\grad\phi}_\infty^2 N_1^{-1} e^{-\frac{C_{\text{rate,1}}(2)}{2} t} \\
    & \quad + 6\sigma_1^2 (1+R_\cut^2) (1 + e^{2\alpha C_\cE (1+2R_\cut^2)}) \norm{\grad\phi}_\infty^2 e^{-\kappa t} \E\cD^X_{N_1}(t) \\
    & \quad + 6\sigma_1^2 C_{\text{MZ,2}} e^{2\alpha (C_{\overline{\cE}} + C_{\underline{\cE}})(1+R_\cut^2)} (1 + e^{2\alpha C_\cE(1+2R_\cut^2)}) R_\cut^2 N_1^{-1} e^{-C_{\text{rate,1}}(1)t}  \\
    & \quad + 6\sigma_1^2 e^{4\alpha C_{\overline{\cE}}(1+R_\cut^2)} L_\cE^2 (2R_\cut)^4 (1+4R_\cut)^2 N_2^{-1} e^{-\frac{C_{\text{rate,2}}(2)}{2} t} \,.
\end{align*}

\step[Summary]
Merging the above estimates, we have
\begin{align*}
    \frac{d \E\cD^X_{N_1}(t)}{dt} & \le 2\lambda_1 e^{-\zeta t}  \E \cD^X_{N_1}(t) \\
    & \quad + (2\lambda_1+6\sigma_1^2) C_{\text{MZ,2}} e^{2\alpha (C_{\overline{\cE}} + C_{\underline{\cE}})(1+R_\cut^2)} (1 + e^{2\alpha C_\cE(1+2R_\cut^2)}) R_\cut^2 N_1^{-1} e^{-(C_{\text{rate,1}}(1)-\zeta)t}  \\
    & \quad + 6\sigma_1^2 \sqrt{C_{\text{tail,1}}(4)} (2R_\cut)^8 (1 + e^{2\alpha C_\cE (1+2R_\cut^2)}) \norm{\grad\phi}_\infty^2 N_1^{-1} e^{-\frac{C_{\text{rate,1}}(2)}{2} t} \\
    & \quad + (2\lambda_1+6\sigma_1^2) C_{\text{MZ,4}} e^{4\alpha C_{\overline{\cE}}(1+R_\cut^2)} L_\cE^2 (2R_\cut)^4 (1+4R_\cut)^2 N_2^{-1} e^{-(\frac{C_{\text{rate,2}}(2)}{2} - \zeta) t} \\
    & \quad + (\lambda_1+6\sigma_1^2) C_{\cB} N_1^{-1} e^{-(\frac{C_{\text{rate,1}}(4)}{4} - \zeta)t} \\
    & \quad + 4(\lambda_1+6\sigma_1^2) C_M^2 (1 + R_\cut^2) e^{-(\kappa-\zeta) t}  \E[\cD^X_{N_1}(t) + \cD^Y_{N_2}(t)] \\
    & \quad + 6\sigma_1^2 (1+R_\cut^2) (1 + e^{2\alpha C_\cE (1+2R_\cut^2)}) \norm{\grad\phi}_\infty^2 e^{-\kappa t} \E\cD^X_{N_1}(t)
\end{align*}
with $\kappa = \frac{C_{\text{rate,1}}(4) \land C_{\text{rate,2}}(4)}{8}$ and $\zeta = \frac{\kappa}{2}$.

By symmetry, we also obtain the decay estimate for $\cD^Y_{N_2}$,
\begin{align*}
    \frac{d \E\cD^Y_{N_2}(t)}{dt} & \le 2\lambda_2 e^{-\zeta t}  \E \cD^Y_{N_2}(t) \\
    & \quad + (2\lambda_2+6\sigma_2^2) C_{\text{MZ,2}} e^{2\beta (C_{\overline{\cE}} + C_{\underline{\cE}})(1+R_\cut^2)} (1 + e^{2\beta C_\cE(1+2R_\cut^2)}) R_\cut^2 N_2^{-1} e^{-(C_{\text{rate,2}}(1)-\zeta)t}  \\
    & \quad + 6\sigma_2^2 \sqrt{C_{\text{tail,2}}(4)} (2R_\cut)^8 (1 + e^{2\beta C_\cE (1+2R_\cut^2)}) \norm{\grad\phi}_\infty^2 N_2^{-1} e^{-\frac{C_{\text{rate,2}}(2)}{2} t} \\
    & \quad + (2\lambda_2+6\sigma_2^2) C_{\text{MZ,4}} e^{4\beta C_{\overline{\cE}}(1+R_\cut^2)} L_\cE^2 (2R_\cut)^4 (1+4R_\cut)^2 N_1^{-1} e^{-(\frac{C_{\text{rate,1}}(2)}{2} - \zeta) t} \\
    & \quad + (\lambda_2+6\sigma_2^2) C_{\cB}' N_2^{-1} e^{-(\frac{C_{\text{rate,2}}(4)}{4} - \zeta)t} \\
    & \quad + 4(\lambda_2+6\sigma_2^2) C_M'^2 (1 + R_\cut^2) e^{-(\kappa-\zeta) t}  \E[\cD^X_{N_1}(t) + \cD^Y_{N_2}(t)] \\
    & \quad + 6\sigma_2^2 (1+R_\cut^2) (1 + e^{2\beta C_\cE (1+2R_\cut^2)}) \norm{\grad\phi}_\infty^2 e^{-\kappa t} \E\cD^Y_{N_2}(t) \,,
\end{align*}
with the same $\kappa$ and $\zeta$ as above, but $C_M' = 3\beta e^{4\beta C_\cE(1+R_\cut^2)}$ and $C_{\cB}' = 2^{11}C_M'^2 R_\cut^8 \sqrt{C_{\text{tail,2}}(4) + \bar C_{\text{tail,2}}(4)}$.

Let $\bar\lambda = \lambda_1 \lor \lambda_2$, $\bar\sigma = \sigma_1 \lor \sigma_2$, and $\gamma = \alpha\lor\beta$.
Merging the above inequalities, we reach~\eqref{eq:differential}
\begin{equation*}
    \frac{d \E[\cD^X_{N_1}(t) + \cD^Y_{N_2}(t)]}{dt} \le C_{\text{decay}} e^{-\zeta t} \E[\cD^X_{N_1}(t) + \cD^Y_{N_2}(t)] + C_{\text{error}} e^{-\zeta t} (N_1^{-1} + N_2^{-1})
\end{equation*}
with
\begin{equation*}
    C_{\text{decay}}  = 2\bar\lambda + 72(\bar\lambda+6\bar\sigma^2) \gamma^2 e^{8\gamma C_\cE (1+R_\cut^2)} (1+R_\cut^2) + 6\bar\sigma^2 (1+R_\cut^2) (1+e^{2\gamma C_\cE(1+2R_\cut^2)}) \norm{\grad\phi}_\infty^2 
\end{equation*}
and
\begin{equation*}
\begin{aligned}
    C_{\text{error}} & = (2\bar\lambda+6\bar\sigma^2) C_{\text{MZ,2}} e^{2\gamma (C_{\overline{\cE}} + C_{\underline{\cE}})(1+R_\cut^2)} (1 + e^{2\gamma C_\cE(1+2R_\cut^2)}) R_\cut^2   \\
    & \qquad + 6\bar\sigma^2 \sqrt{C_{\text{tail,1}}(4) \lor C_{\text{tail,2}}(4)} (2R_\cut)^8 (1 + e^{2\gamma C_\cE (1+2R_\cut^2)}) \norm{\grad\phi}_\infty^2  \\
    & \qquad + (2\bar\lambda+6\bar\sigma^2) C_{\text{MZ,4}} e^{4\gamma C_{\overline{\cE}}(1+R_\cut^2)} L_\cE^2 (2R_\cut)^4 (1+4R_\cut)^2  \\
    & \qquad + 18432 (\bar\lambda+6\bar\sigma^2) \gamma^2e^{8\gamma C_\cE(1+R_\cut^2)} R_\cut^8 \times \\ 
    & \qquad \qquad (\sqrt{C_{\text{tail,1}}(4) + \bar C_{\text{tail,1}}(4)} \lor \sqrt{C_{\text{tail,2}}(4) + \bar C_{\text{tail,2}}(4)})   \,.
\end{aligned}
\end{equation*}
Applying Grönwall's inequality, we conclude that
\begin{equation*}
    \E[\cD^X_{N_1}(t) + \cD^Y_{N_2}(t)] \le \left( \E[\cD^X_{N_1}(0) + \cD^Y_{N_2}(0)] + C_{\text{error}}\zeta^{-1} (N_1^{-1} + N_2^{-1}) \right) e^{\frac{C_{\text{decay}}}{\zeta}} \,.
\end{equation*}

\begin{remark}
    From the proof above, one might notice that the constant multiple $C_{\text{error}} \zeta^{-1} e^{\frac{C_{\text{decay}}}{\zeta}}$ grows double exponentially with the Laplace scalar $\alpha$ and $\beta$.
    Consensus-based algorithms usually run with relatively large $\alpha$ and $\beta$ to guarantee the proximity of the limit point to the true optimum.
    This gives rise to a giant constant term in the asymptotic analysis, and it is worth looking for improvement.
\end{remark}

\section{Proof of auxiliary results}

\subsection{Lemmata on generalized variance}
\label{s:prf-var}

The variance of the law under the mean-reverting dynamics of the consensus-based algorithms shows rapid decays.
We first prove the exponential convergence of the expected variance claimed in Section~\ref{s:var}.

\begin{proof}[Proof of Lemma~\ref{lm:var-decay-particles}]
For simplicity, we neglect subscript of $D$ in this proof.
Observe that
\begin{align*}
    d M(\nu_{\bX^{N_1}_t}) & = -\lambda_1 (M(\nu_{\bX^{N_1}_t}) - \cM_\alpha(\nu_{\bX^{N_1}_t}, \nu_{\bY^{N_2}_t})) dt \\
    & \qquad + \frac{\sigma_1}{N_1} \sum_{j=1}^{N_1} D(X^{N_1,j}_t, \cM_\alpha(\nu_{\bX^{N_1}_t}, \nu_{\bY^{N_2}_t})) dW^{X,j}_t \,,
\end{align*}
so
\begin{align*}
    d (X^{N_1,k}_t - M(\nu_{\bX^{N_1}_t})) & = -\lambda_1 (X^{N_1,k}_t - M(\nu_{\bX^{N_1}_t})) dt \\
    & \qquad + \sigma_1(1-N_1^{-1}) D(X^{N_1,k}_t, \cM_\alpha(\nu_{\bX^{N_1}_t}, \nu_{\bY^{N_2}_t})) dW^{X,k}_t \\
    & \qquad - \frac{\sigma_1}{N_1} \sum_{j\neq k} D(X^{N_1,j}_t, \cM_\alpha(\nu_{\bX^{N_1}_t}, \nu_{\bY^{N_2}_t})) dW^{X,j}_t \,.
\end{align*}
Then, for $q=1$,
\begin{align*}
    d \abs{X^{N_1,k}_t - M(\nu_{\bX^{N_1}_t})}^2 & = -2\lambda_1 \abs{X^{N_1,k}_t - M(\nu_{\bX^{N_1}_t})}^2 dt  \\
    & \quad + \sigma_1^2 (1-N_1^{-1})^2 \tr[D^\top D (X^{N_1,k}_t, \cM_\alpha(\nu_{\bX^{N_1}_t}, \nu_{\bY^{N_2}_t}))] dt \\
    & \qquad + \sigma_1^2 N_1^{-2} \sum_{j \neq k} \tr[D^\top D (X^{N_1,j}_t, \cM_\alpha(\nu_{\bX^{N_1}_t}, \nu_{\bY^{N_2}_t}))] dt \\
    & \quad + 2\sigma_1 (X^{N_1,k}_t - M(\nu_{\bX^{N_1}_t}))^\top D(X^{N_1,k}_t, \cM_\alpha(\nu_{\bX^{N_1}_t}, \nu_{\bY^{N_2}_t})) dW^{X,k}_t \\
    & \qquad - \frac{2\sigma_1}{N_1} \sum_{j=1}^{N_1} (X^{N_1,k}_t - M(\nu_{\bX^{N_1}_t}))^\top D(X^{N_1,j}_t, \cM_\alpha(\nu_{\bX^{N_1}_t}, \nu_{\bY^{N_2}_t})) dW^{X,j}_t \,.
\end{align*}
While for $q \ge 2$, we have
\begin{align*}
    & d \abs{X^{N_1,k}_t - M(\nu_{\bX^{N_1}_t})}^{2q}  = -2q \lambda_1 \abs{X^{N_1,k}_t - M(\nu_{\bX^{N_1}_t})}^{2q} dt  \\
    & \quad + q(2q-2) \sigma_1^2 (1-N_1^{-1})^2 \abs{X^{N_1,k}_t - M(\nu_{\bX^{N_1}_t})}^{2q-4} \times \\
    & \qquad\qquad\qquad \abs{D (X^{N_1,k}_t, \cM_\alpha(\nu_{\bX^{N_1}_t}, \nu_{\bY^{N_2}_t}))  (X^{N_1,k}_t - M(\nu_{\bX^{N_1}_t}))}^2 dt \\
    & \qquad + q(2q-2) \sigma_1^2 N_1^{-2} \sum_{j \neq k} \abs{X^{N_1,k}_t - M(\nu_{\bX^{N_1}_t})}^{2q-4} \times \\
    & \qquad\qquad\qquad \abs{D (X^{N_1,j}_t, \cM_\alpha(\nu_{\bX^{N_1}_t}, \nu_{\bY^{N_2}_t}))  (X^{N_1,k}_t - M(\nu_{\bX^{N_1}_t}))}^2 dt \\
    & \quad + q \sigma_1^2 (1-N_1^{-1})^2 \abs{X^{N_1,k}_t - M(\nu_{\bX^{N_1}_t})}^{2q-2} \tr[D^\top D (X^{N_1,k}_t , \cM_\alpha(\nu_{\bX^{N_1}_t}, \nu_{\bY^{N_2}_t}))] dt \\
    & \qquad + q \sigma_1^2 N_1^{-2} \sum_{j \neq k} \abs{X^{N_1,k}_t - M(\nu_{\bX^{N_1}_t})}^{2q-2} \tr[D^\top D (X^{N_1,j}_t, \cM_\alpha(\nu_{\bX^{N_1}_t}, \nu_{\bY^{N_2}_t}))] dt \\
    & \quad + 2q \sigma_1 \abs{X^{N_1,k}_t - M(\nu_{\bX^{N_1}_t})}^{2q-2} (X^{N_1,k}_t - M(\nu_{\bX^{N_1}_t}))^\top D(X^{N_1,k}_t, \cM_\alpha(\nu_{\bX^{N_1}_t}, \nu_{\bY^{N_2}_t})) dW^{X,k}_t \\
    & \qquad - \frac{2q\sigma_1}{N_1} \sum_{j=1}^{N_1} \abs{X^{N_1,k}_t - M(\nu_{\bX^{N_1}_t})}^{2q-2} (X^{N_1,k}_t - M(\nu_{\bX^{N_1}_t}))^\top D(X^{N_1,j}_t, \cM_\alpha(\nu_{\bX^{N_1}_t}, \nu_{\bY^{N_2}_t})) dW^{X,j}_t \,.
\end{align*}
Summing over $k$ and taking expectation, we have
\begin{align*}
    & d \E \cV_{2q}(\nu_{\bX^{N_1}_t}) \le \\
    & \quad -2q\lambda_1 \E\cV_{2q}(\nu_{\bX^{N_1}_t}) dt  \\
    & \quad + q(2q-1) \sigma_1^2 (1-N_1^{-1})^2 N_1^{-1} \sum_{k=1}^{N_1} \E \left[ \abs{X^{N_1,k}_t - M(\nu_{\bX^{N_1}_t})}^{2q-2} \abs{X^{N_1,k}_t - \cM_\alpha(\nu_{\bX^{N_1}_t}, \nu_{\bY^{N_2}_t})}^2 \right] \\
    & \quad + q(2q-1) \sigma_1^2 N_1^{-3} \sum_{k=1}^{N_1} \sum_{j \neq k} \E \left[ \abs{X^{N_1,k}_t - M(\nu_{\bX^{N_1}_t})}^{2q-2} \abs{X^{N_1,j}_t - \cM_\alpha(\nu_{\bX^{N_1}_t}, \nu_{\bY^{N_2}_t})}^2 \right] \\
    & = -2q\lambda_1 \E \cV_{2q}(\nu_{\bX^{N_1}_t}) dt \\
    & \quad + q(2q-1) \sigma_1^2 (1-2N_1^{-1}) N_1^{-1} \sum_{k=1}^{N_1} \E \left[ \abs{X^{N_1,k}_t - M(\nu_{\bX^{N_1}_t})}^{2q-2} \abs{X^{N_1,k}_t - \cM_\alpha(\nu_{\bX^{N_1}_t}, \nu_{\bY^{N_2}_t})}^2 \right] \\
    & \quad + q(2q-1) \sigma_1^2 N_1^{-3} \sum_{k=1}^{N_1} \sum_{j=1}^{N_1} \E \left[ \abs{X^{N_1,k}_t - M(\nu_{\bX^{N_1}_t})}^{2q-2} \abs{X^{N_1,j}_t - \cM_\alpha(\nu_{\bX^{N_1}_t}, \nu_{\bY^{N_2}_t})}^2 \right] \,.
\end{align*}

By Lemma~\ref{l:mean-consensus},
\begin{equation*}
    \abs{M(\nu_{\bX^{N_1}_t}) - \cM_\alpha(\nu_{\bX^{N_1}_t}, \nu_{\bY^{N_2}_t})}^2 \le e^{2\alpha C_\cE(1+2R_{cut}^2)} \var{\nu_{\bX^{N_1}_t}} \,.
\end{equation*}
Applying it to the evolution of $\cV_{2q}$, we get
\begin{align*}
    & d \E\cV_{2q}(\nu_{\bX^{N_1}_t}) \le \\
    & \quad -2q \lambda_1 \E\cV_{2q}(\nu_{\bX^{N_1}_t}) dt  \\
    & \quad + 2q(2q-1) \sigma_1^2 e^{2\alpha C_\cE(1+2R_{cut}^2)} (1-2N_1^{-1}) N_1^{-1} \sum_{k=1}^{N_1} \E \left[ \abs{X^{N_1,k}_t - M(\nu_{\bX^{N_1}_t})}^{2q-2}  \var{\nu_{\bX^{N_1}_t}} \right] \\
    & \quad + 2q(2q-1) \sigma_1^2 (1-2N_1^{-1}) N_1^{-1} \sum_{k=1}^{N_1} \E \abs{X^{N_1,k}_t - M(\nu_{\bX^{N_1}_t})}^{2q} \\
    & \quad + 2q(2q-1) \sigma_1^2 (1+e^{2\alpha C_\cE(1+2R_{cut}^2)}) N_1^{-2} \sum_{k=1}^{N_1} \E \left[ \abs{X^{N_1,k}_t - M(\nu_{\bX^{N_1}_t})}^{2q-2} \var{\nu_{\bX^{N_1}_t}} \right] \,.
\end{align*}
By Jensen's inequality, we have
\begin{equation*}
    N_1^{-1} \sum_{k=1}^{N_1} \abs{X^{N_1,k}_t - M(\nu_{\bX^{N_1}_t})}^{2q-2} \le \left( N_1^{-1} \sum_{k=1}^{N_1} \abs{X^{N_1,k}_t - M(\nu_{\bX^{N_1}_t})}^{2q} \right)^{\frac{q-1}{q}} = \cV_{2q}(\nu_{\bX^{N_1}_t})^{\frac{q-1}{q}} \,,
\end{equation*}
and
\begin{equation*}
    \var{\nu_{\bX^{N_1}_t}} = \cV_2(\nu_{\bX^{N_1}_t}) \le \cV_{2q}(\nu_{\bX^{N_1}_t})^{\frac 1q} \,.
\end{equation*}
Plugging back into the above inequalities, we obtain
\begin{align*}
    \frac{d}{dt} \E \cV_{2q}(\nu_{\bX^{N_1}_t}) & \le -2q\lambda_1 \E \cV_{2q}(\nu_{\bX^{N_1}_t}) \\
    & + 2q(2q-1) \sigma_1^2 (1-N_1^{-1})(1 + e^{2\alpha C_\cE(1+2R_\cut^2)}) \E \cV_{2q}(\nu_{\bX^{N_1}_t}) \,.
\end{align*}
Symmetrically, the following argument holds for the Y-particles as well,
\begin{align*}
    \frac{d}{dt} \E \cV_{2q}(\nu_{\bY^{N_2}_t}) & \le -2q\lambda_2 \E \cV_{2q}(\nu_{\bY^{N_2}_t}) \\
    & + 2q(2q-1) \sigma_2^2 (1-N_2^{-1})(1 + e^{2\beta C_\cE(1+2R_\cut^2)}) \E \cV_{2q}(\nu_{\bY^{N_2}_t}) \,.
\end{align*}
\end{proof}

\begin{proof}[Proof of Lemma~\ref{lm:var-decay-meanfield}]
The strategy of the proof is identical to that of Lemma~\ref{lm:var-decay-particles}.
It is even more succinct since it involves one single process.

Let $(\bar X_t, \bar Y_t)_{t \ge 0}$ be the canonical process of the law $(\bar\nu^X_t \otimes \bar\nu^Y_t)_{t \ge 0}$.
It suffices to work on $\bar X_t$.

Notice that $M(\bar\nu^X_t) = \E \bar X_t$, so
\begin{equation*}
    d (\bar X_t - M(\bar\nu^X_t)) = -\lambda_1 (\bar X_t - M(\bar\nu^X_t)) dt + \sigma_1 D(\bar X_t, \cM_\alpha(\bar\nu^X_t, \bar\nu^Y_t)) dW^X_t \,,
\end{equation*}
which gives
\begin{align*}
    d \abs{\bar X_t - M(\bar\nu^X_t)}^{2q} & \le -2q\lambda_1 \abs{\bar X_t - M(\bar\nu^X_t)}^{2q} dt \\
    & \quad + q(2q-1) \sigma_1^2 \abs{\bar X_t - M(\bar\nu^X_t)}^{2q-2} \abs{\bar X_t - \cM_\alpha(\bar\nu^X_t, \bar\nu^Y_t)}^2 dt \\
    & \quad + 2q \sigma_1 \abs{\bar X_t - M(\bar\nu^X_t)}^{2q-2} (\bar X_t - M(\bar\nu^X_t))^\top D(\bar X_t, \cM_\alpha(\bar\nu^X_t, \bar\nu^Y_t)) dW^X_t \,.
\end{align*}
Along with Lemma~\ref{l:mean-consensus}, we get
\begin{align*}
    \frac{d}{dt} \E \abs{\bar X_t - \cM_\alpha(\bar\nu^X_t, \bar\nu^Y_t)}^{2q} & \le -2q\lambda_1 \E\abs{\bar X_t - \cM_\alpha(\bar\nu^X_t, \bar\nu^Y_t)}^{2q}  \\
    & \quad + 2q(2q-1) \sigma_1^2 (1 + e^{2\alpha C_\cE (1 + 2R_\cut^2)}) \E\abs{\bar X_t - \cM_\alpha(\bar\nu^X_t, \bar\nu^Y_t)}^{2q}  \,.
\end{align*}
Thus
\begin{equation*}
    \cV_{2q}(\bar\nu^X_t) \le e^{-C_{\text{rate,1}}(q) t} \cV_{2q}(\bar\nu^X_0) \,.
\end{equation*}
The symmetric argument holds for $\bar\nu^Y_t$.
\end{proof}

The variances of the particle systems $(\bX^{N_1}, \bY^{N_2})$ and $(\bar\bX^{N_1}, \bar\bY^{N_2})$ also show high concentration.
We now prove the concentration inequalities Lemmata~\ref{lm:tail-bound-particles} and~\ref{lm:tail-bound-copies}.

\begin{proof}[Proof of Lemma~\ref{lm:tail-bound-particles}]
We only write the proof for the first inequality.

Recall from the previous proofs that
\begin{align*}
    d \var{\nu_{\bX^{N_1}_t}} & = -2\lambda_1 \var{\nu_{\bX^{N_1}_t}} dt \\
    & \quad + \sigma_1^2 (1-N_1^{-1})^2 N_1^{-1} \sum_{k=1}^{N_1} \tr[D^\top D (X^{N_1,k}_t, \cM_\alpha(\nu_{\bX^{N_1}_t}, \nu_{\bY^{N_2}_t}))] dt \\
    & \quad + \sigma_1^2 N_1^{-3} \sum_{k=1}^{N_1} \sum_{j \neq k} \tr[D^\top D (X^{N_1,j}_t, \cM_\alpha(\nu_{\bX^{N_1}_t}, \nu_{\bY^{N_2}_t}))] dt \\
    & \quad + 2\sigma_1 N_1^{-1} \sum_{k=1}^{N_1} (X^{N_1,k}_t - M(\nu_{\bX^{N_1}_t}))^\top D (X^{N_1,k}_t, \cM_\alpha(\nu_{\bX^{N_1}_t}, \nu_{\bY^{N_2}_t})) dW^{X,k}_t \\
    & \quad - 2\sigma_1 N_1^{-2} \sum_{k=1}^{N_1} \sum_{j=1}^{N_1} (X^{N_1,k}_t - M(\nu_{\bX^{N_1}_t}))^\top D (X^{N_1,j}_t, \cM_\alpha(\nu_{\bX^{N_1}_t}, \nu_{\bY^{N_2}_t})) dW^{X,j}_t \,.
\end{align*}
Choosing $\kappa < C_{\text{rate,1}}(1)/2$, we have
\begin{equation*}
    e^{\kappa t} \var{\nu_{\bX^{N_1}_t}} \le \var{\nu_{\bX^{N_1}_0}} + M_t \,,
\end{equation*}
where $M$ is a martingale defined by
\begin{align*}
    dM_t & = 2\sigma_1 e^{\kappa t} N_1^{-1} \sum_{k=1}^{N_1} (X^{N_1,k}_t - M(\nu_{\bX^{N_1}_t}))^\top D (X^{N_1,k}_t, \cM_\alpha(\nu_{\bX^{N_1}_t}, \nu_{\bY^{N_2}_t})) dW^{X,k}_t \\
    & \quad - 2\sigma_1 e^{\kappa t} N_1^{-2} \sum_{k=1}^{N_1} \sum_{j=1}^{N_1} (X^{N_1,k}_t - M(\nu_{\bX^{N_1}_t}))^\top D (X^{N_1,j}_t, \cM_\alpha(\nu_{\bX^{N_1}_t}, \nu_{\bY^{N_2}_t})) dW^{X,j}_t
\end{align*}
with $M_0 = 0$.
Observe that the second line vanishes.
We will show that
\begin{equation*}
    t \mapsto \E \left[ \abs{D (X^{N_1,k}_t, \cM_\alpha(\nu_{\bX^{N_1}_t}, \nu_{\bY^{N_2}_t})) (X^{N_1,k}_t - M(\nu_{\bX^{N_1}_t}))}^q \right]
\end{equation*}
is in $L^1(0,\infty)$.

We first apply Hölder's inequality and Minkowski's inequality to see that
\begin{align*}
    & \E \left[ \abs{D (X^{N_1,k}_t, \cM_\alpha(\nu_{\bX^{N_1}_t}, \nu_{\bY^{N_2}_t})) (X^{N_1,k}_t - M(\nu_{\bX^{N_1}_t}))}^q \right] \\
    & \le \left( \E \abs{X^{N_1,k}_t - \cM_\alpha(\nu_{\bX^{N_1}_t}, \nu_{\bY^{N_2}_t})}^{2q} \right)^{1/2} \left( \E \abs{X^{N_1,k}_t - M(\nu_{\bX^{N_1}_t})}^{2q} \right)^{1/2} \\
    & \le 2^{q-1} \left( \E \abs{M(\nu_{\bX^{N_1}_t}) - \cM_\alpha(\nu_{\bX^{N_1}_t}, \nu_{\bY^{N_2}_t})}^{2q} \right)^{1/2} \left( \E \abs{X^{N_1,k}_t - M(\nu_{\bX^{N_1}_t})}^{2q} \right)^{1/2} \\
    & \qquad + 2^{q-1} \E \abs{X^{N_1,k}_t - M(\nu_{\bX^{N_1}_t})}^{2q}
\end{align*}
By symmetry,
\begin{equation*}
    \E \abs{X^{N_1,k}_t - M(\nu_{\bX^{N_1}_t})}^{2q} = \frac{1}{N_1} \sum_{j=1}^{N_1} \E \abs{X^{N_1,j}_t - M(\nu_{\bX^{N_1}_t})}^{2q} = \E \cV_{2q}(\nu_{\bX^{N_1}_t}) \,.
\end{equation*}
In addition, by Lemma~\ref{l:mean-consensus}
, we have
\begin{equation*}
    \E \abs{M(\nu_{\bX^{N_1}_t}) - \cM_\alpha(\nu_{\bX^{N_1}_t}, \nu_{\bY^{N_2}_t})}^{2q} \le e^{2\alpha C_\cE (1+2R_\cut^2)} \E \cV_{2q}(\nu_{\bX^{N_1}_t}) \,.
\end{equation*}
Then
\begin{align*}
    & \E \left[ \abs{D (X^{N_1,k}_t, \cM_\alpha(\nu_{\bX^{N_1}_t}, \nu_{\bY^{N_2}_t})) (X^{N_1,k}_t - M(\nu_{\bX^{N_1}_t}))}^q \right] \\
    & \le 2^{q-1} (1 + e^{2\alpha C_\cE (1+2R_\cut^2)})^{\frac 12} \E \cV_{2q}(\nu_{\bX^{N_1}_t}) \\
    & \le 2^{q-1} (1 + e^{2\alpha C_\cE (1+2R_\cut^2)})^{\frac 12} e^{-C_{\text{rate,1}}(q) t} \E \cV_{2q}(\nu_{\bX^{N_1}_0}) \,.
\end{align*}
When $C_{\text{rate},1}(q) > 2\kappa$, it allows us to adopt Lemma~4.8 from~\cite{GerberHoffmannKimVaes2025}:
for all $q \ge 2$, $A > 0$, and $t > 0$, we have
\begin{align}
\nonumber
    & \P \left[ \sup_{s \in [0,t]} e^{\kappa s} \var{\nu_{\bX^{N_1}_s}} \ge \E \var{\nu_{\bX^{N_1}_0}} + A \right] \le \\
    \label{e:tail-A}
    & \quad \frac{2^q}{A^q} \E \left[ \abs{\var{\nu_{\bX^{N_1}_0}} - \E \var{\nu_{\bX^{N_1}_0}}}^q \right] + \frac{2^q}{A^q} \E \left[ \sup_{s \in [0,t]} \abs{M_s}^q \right] \,.
\end{align}

We now follow the proof of Lemma~4.9 in~\cite{GerberHoffmannKimVaes2025} to see that the first term in~\eqref{e:tail-A} can be bounded by
\begin{equation*}
    \E \left[ \abs{\var{\nu_{\bX^{N_1}_0}} - \E \var{\nu_{\bX^{N_1}_0}}}^q \right] \le 2^{2q} C_{\text{MZ},2q} N_1^{-\frac q2} \cV_{2q}(\bar \nu^X_0) \,,
\end{equation*}
where $C_{\text{MZ,}2q}$ is the constant rising from the MZ inequality.
For the second term of~\eqref{e:tail-A}, we apply BDG inequality and Hölder's inequality to see that
\begin{align*}
    & \E \left[ \sup_{s \in [0,t]} \abs{M_s}^q \right] \le C_{\text{BDG},q} \E [\ip{M}_t^{\frac q2}] \\
    & \quad \le C_{\text{BDG},q} (2\sigma_1)^q N_1^{-1-\frac q2} \sum_{k=1}^{N_1} \ell^{1-\frac q2} \int_0^t e^{(\frac q2 - 1) \ell s} e^{q\kappa s} \times \\
    & \qquad \qquad \E \abs{D (X^{N_1,k}_s, \cM_\alpha(\nu_{\bX^{N_1}_s}, \nu_{\bY^{N_2}_s})) (X^{N_1,k}_s - M(\nu_{\bX^{N_1}_s}))}^q ds \\
    & \quad \le C_{\text{BDG},q} (2\sigma_1)^q N_1^{-\frac q2} \ell^{1-\frac q2} \int_0^t e^{(\frac q2 - 1) \ell s + q\kappa s} 2^{q-1} (1 + e^{2\alpha C_\cE(1+2R_\cut^2)})^{\frac 12} e^{-C_{\text{rate,1}}(q) s} \E \cV_{2q}(\nu_{\bX^{N_1}_0})  ds \,,
\end{align*}
where $\ip{M}$ is the quadratic variation of $M$.
Taking $\ell = \frac{C_{\text{rate},1}(q) - q\kappa}{q-2}$, we have
\begin{align*}
    & \E \left[ \sup_{s \in [0,t]} \abs{M_s}^q \right] \le \\
    & \quad  C_{\text{BDG},q} (2\sigma_1)^q N_1^{-\frac q2}  \ell^{1-\frac q2} 2^{q-1} (1 + e^{2\alpha C_\cE(1+2R_\cut^2)})^{\frac 12} \int_0^t e^{-\frac{C_{\text{rate},1}(q) - q\kappa}{2} s} (\E\cV_{2q}(\nu_{\bX^{N_1}_0})) ds
\end{align*}
Joining the above bounds gives
\begin{equation*}
    \P \left[ \sup_{s \in [0,\infty]} e^{\kappa s} \var{\nu_{\bX^{N_1}_s}} \ge \E\var{\nu_{\bX^{N_1}_0}} + A \right] \le C_{\text{tail,1}}(q) A^{-q}  N_1^{-\frac q2} \cV_{2q}(\bar \nu^X_0) \,,
\end{equation*}
with
\begin{equation*}
    C_{\text{tail,1}}(q) = 2^{3q} C_{\text{MZ},2q} + 2^{3q} \sigma_1^q C_{\text{BDG},q} (C_{\text{rate,1}}(q) - q\kappa)^{-\frac q2} (q-2)^{-\frac{q-2}{2}} (1 + e^{2\alpha C_\cE(1+2R_\cut^2)})^{\frac 12} \,.
\end{equation*}

From the proof of Corollary~\ref{cr:var-decay-copies}, we also see that $\E\var{\nu_{\bX^{N_1}_0}} \le \var{\bar\nu^X_0}$.
Thus the above result also leads to
\begin{equation*}
    \P \left[ \sup_{s \in [0,\infty]} e^{\kappa s} \var{\nu_{\bX^{N_1}_s}} \ge \var{\bar\nu^X_0} + A \right] \le C_{\text{tail,1}}(q) A^{-q} N_1^{-\frac q2} \cV_{2q}(\bar \nu^X_0) \,.
\end{equation*}
Analogously, we also have
\begin{equation*}
    \P \left[ \sup_{s \in [0,\infty]} e^{\kappa s} \var{\nu_{\bY^{N_2}_s}} \ge \var{\bar\nu^Y_0} + A \right] \le C_{\text{tail,2}}(q) A^{-q} N_2^{-\frac q2} \cV_{2q}(\bar \nu^Y_0) \,.
\end{equation*}
\end{proof}

\begin{proof}[Proof of Lemma~\ref{lm:tail-bound-copies}]
Observe that
\begin{equation*}
    \var{\nu_{\bar\bX^{N_1}_t}} = \frac{1}{N_1} \sum_{k=1}^{N_1} \abs{\bar X^k_t - M(\nu_{\bar\bX^{N_1}_t})}^2 \le \frac{1}{N_1} \sum_{k=1}^{N_1} \abs{\bar X^k_t - M(\bar\nu^X_t)}^2 = \frac{1}{N_1} \sum_{k=1}^{N_1} \abs{\bar X^k_t - \E \bar X^k_t}^2
\end{equation*}
since the mean minimizes the averaged Euclidean distance.
It suffices to prove the concentration inequality with $\var{\nu_{\bar\bX^{N_1}_t}}$ replaced by the right-hand side above.

For each $k \in [N_1]$,
\begin{align*}
    d \abs{\bar X^k_t - \E \bar X^k_t}^2 & = -2\lambda_1 \abs{\bar X^k_t - \E \bar X^k_t}^2 dt + \sigma_1^2 \tr[ D^\top D (\bar X^k_t, \cM_\alpha(\bar\nu^X_t, \bar\nu^Y_t)) ] dt \\
    & \quad +2\sigma_1 (\bar X^k_t - \E \bar X^k_t)^\top D (\bar X^k_t, \cM_\alpha(\bar\nu^X_t, \bar\nu^Y_t)) dW^{X,k}_t \,,
\end{align*}
Then
\begin{align*}
    d \left( \frac{1}{N_1} \sum_{k=1}^{N_1} \abs{\bar X^k_t - \E \bar X^k_t}^2 \right) & \le - \frac{2\lambda_1}{N_1} \sum_{k=1}^{N_1} \abs{\bar X^k_t - \E \bar X^k_t}^2 dt + \frac{\sigma_1^2}{N_1} \sum_{k=1}^{N_1} \abs{\bar X^k_t - \cM_\alpha(\bar\nu^X_t, \bar\nu^Y_t)}^2 \\
    & \qquad + \frac{2\sigma_1}{N_1} \sum_{k=1}^{N_1} (\bar X^k_t - \E \bar X^k_t)^\top D (\bar X^k_t, \cM_\alpha(\bar\nu^X_t, \bar\nu^Y_t)) dW^{X,k}_t \,.
\end{align*}
We may further expand the second term and apply Lemma~\ref{l:mean-consensus} to obtain
\begin{align*}
    \left( \frac{1}{N_1} \sum_{k=1}^{N_1} \abs{\bar X^k_t - \cM_\alpha(\bar\nu^X_t, \bar\nu^Y_t)}^2 \right)^{\frac 12} & \le \left( \frac{1}{N_1} \sum_{k=1}^{N_1} \abs{\bar X^k_t - \E \bar X^k_t}^2 \right)^{\frac 12} + \abs{M(\nu_{\bar\bX^{N_1}_t}) - \cM_\alpha(\nu_{\bar\bX^{N_1}_t}, \bar\nu^Y_t)}  \\
    & \quad + \abs{M(\bar\nu^X_t) - M(\nu_{\bar\bX^{N_1}_t}) - \cM_\alpha(\bar\nu^X_t, \bar\nu^Y_t) + \cM_\alpha(\nu_{\bar\bX^{N_1}_t}, \bar\nu^Y_t)} \\
    & \le (1 + e^{\alpha C_\cE(1+2R_\cut^2)}) \left( \frac{1}{N_1} \sum_{k=1}^{N_1} \abs{\bar X^k_t - \E \bar X^k_t}^2 \right)^{\frac 12} \\
    & \quad + \abs{M(\bar\nu^X_t) - M(\nu_{\bar\bX^{N_1}_t}) - \cM_\alpha(\bar\nu^X_t, \bar\nu^Y_t) + \cM_\alpha(\nu_{\bar\bX^{N_1}_t}, \bar\nu^Y_t)} \,.
\end{align*}
Then, for any $\ep > 0$, we have
\begin{align*}
    & \frac{1}{N_1} \sum_{k=1}^{N_1} \abs{\bar X^k_t - \cM_\alpha(\bar\nu^X_t, \bar\nu^Y_t)}^2 \le \\
    & \quad (1+\ep) (1 + e^{\alpha C_\cE(1+2R_\cut^2)})^2 \left( \frac{1}{N_1} \sum_{k=1}^{N_1} \abs{\bar X^k_t - \E \bar X^k_t}^2 \right) \\
    & \quad + (1+\ep^{-1}) \abs{M(\bar\nu^X_t) - M(\nu_{\bar\bX^{N_1}_t}) - \cM_\alpha(\bar\nu^X_t, \bar\nu^Y_t) + \cM_\alpha(\nu_{\bar\bX^{N_1}_t}, \bar\nu^Y_t)}^2 \,.
\end{align*}
Define
\begin{align*}
    \tilde V_t = \frac{e^{\kappa t}}{N_1} \sum_{k=1}^{N_1} \abs{\bar X^k_t - \E \bar X^k_t}^2 \,.
\end{align*}
Taking $\ep = \sigma_1^{-2} (C_{\text{rate,1}}(1)-\kappa) (1 + e^{\alpha C_\cE(1+2R_\cut^2)})^{-2}$,
we have 
\begin{align*}
    \sigma_1^2 (1+\ep) (1 + e^{\alpha C_\cE(1+2R_\cut^2)})^2 & = \sigma_1^2 (1 + e^{\alpha C_\cE(1+2R_\cut^2)})^2 + (C_{\text{rate,1}}(1) - \kappa) \\
    & \le 2\sigma_1^2 (1 + e^{2\alpha C_\cE(1+2R_\cut^2)}) + 2(\lambda_1 -\sigma_1^2 (1 + e^{2\alpha C_\cE(1+2R_\cut^2)})) - \kappa \\
    & = 2\lambda_1 -\kappa \,.
\end{align*}
Then
\begin{equation*}
    \tilde V_t \le \tilde V_0 + Z_t + M_t \,,
\end{equation*}
where $Z$ and $M$ are defined by
\begin{align*}
    d Z_t & = \sigma_1^2 (1+\ep^{-1}) e^{\kappa t} \abs{M(\bar\nu^X_t) - M(\nu_{\bar\bX^{N_1}_t}) - \cM_\alpha(\bar\nu^X_t, \bar\nu^Y_t) + \cM_\alpha(\nu_{\bar\bX^{N_1}_t}, \bar\nu^Y_t)}^2 dt \,,\\
    d M_t & = \frac{2\sigma_1 e^{\kappa t}}{N_1} \sum_{k=1}^{N_1} (\bar X^k_t - \E \bar X^k_t)^\top D (\bar X^k_t, \cM_\alpha(\bar\nu^X_t, \bar\nu^Y_t)) dW^{X,k}_t
\end{align*}
with $Z_0 = 0$ and $M_0 = 0$.
Note that $M$ is a martingale of the same pattern as the one in the proof of Lemma~\ref{lm:tail-bound-particles}.
Following the same spirit, we have
\begin{equation*}
    \E\abs{D (\bar X^k_t, \cM_\alpha(\bar\nu^X_t, \bar\nu^Y_t)) (\bar X^k_t - \E \bar X^k_t)}^q  \le 2^{q-1} (1 + e^{2\alpha C_\cE (1+2R_\cut^2)})^{\frac 12} e^{-C_{\text{rate,1}}(q) t} \E \cV_{2q}(\nu_{\bX^{N_1}_0}) \,,
\end{equation*}
which is integrable in $t \in (0,\infty)$.
By Lemma 4.8, \cite{GerberHoffmannKimVaes2025}, we have
\begin{align*}
    & \E [\sup_{s \in [0,t]} \abs{M_s}^q] \le \\
    & C_{\text{BDG},q} (2\sigma_1)^q N_1^{-\frac q2} \ell^{1-\frac q2} \int_0^t e^{(\frac q2 - 1) \ell s + q\kappa s} 2^{q-1} (1 + e^{2\alpha C_\cE(1+2R_\cut^2)})^{\frac 12} e^{-C_{\text{rate,1}}(q) s} \E \cV_{2q}(\nu_{\bX^{N_1}_0})  ds \\
    & \le 2^{2q} \sigma_1^q C_{\text{BDG},q} (C_{\text{rate,1}}(q) - q\kappa)^{-\frac q2} (q-2)^{-\frac{q-2}{2}} (1 + e^{2\alpha C_\cE(1+2R_\cut^2)})^{\frac 12} N_1^{-\frac q2} \E \cV_{2q}(\nu_{\bX^{N_1}_0}) \,,
\end{align*}
and
\begin{equation}
\label{e:split-V}
    \P \left[ \sup_{s \in [0,t]} \tilde V_s \ge \E \tilde V_0 + A \right] \le \frac{3^q}{A^q} \E \abs{\tilde V_0 - \E \tilde V_0}^q + \frac{3^q}{A^q} \E \left[ \sup_{s \in [0,t]} \abs{Z_s}^q \right] + \frac{3^q}{A^q} \E \left[ \sup_{s \in [0,t]} \abs{M_s}^q \right] \,.
\end{equation}
The first and the third terms can be bounded in the same manner as the proof of Lemma~\ref{lm:tail-bound-particles}, in total by
\begin{equation*}
    \frac{3^q}{2^q} C_{\text{tail,1}}(q) A^{-q} N_1^{-\frac q2} \cV_{2q}(\bar\nu^X_0) \,.
\end{equation*}

For the second term on the right-hand side of~\eqref{e:split-V}, we apply equation (4.6) in~\cite{GerberHoffmannKimVaes2025} to see that
\begin{align*}
    & \sup_{s \in [0,t]} \abs{Z_s}^q \le \\
    & (1+\ep^{-1})^q \sigma_1^{2q} \ell^{1-q} \int_0^t e^{(q-1)\ell s + q\kappa s} \abs{M(\bar\nu^X_t) - M(\nu_{\bar\bX^{N_1}_t}) - \cM_\alpha(\bar\nu^X_t, \bar\nu^Y_t) + \cM_\alpha(\nu_{\bar\bX^{N_1}_t}, \bar\nu^Y_t)}^{2q} ds \,.
\end{align*}
Notice that $M(\nu_{\bar\bX^{N_1}_s})$ is the mean of $N_1$ independent sampling of the law $\bar\nu^X_s$.
The MZ inequality gives
\begin{align*}
    \E \abs{M(\bar\nu^X_s) - M(\nu_{\bar\bX^{N_1}_s})}^{2q} = \E \abs{\frac{1}{N_1} \sum_{k=1}^{N_1} (\E \bar X^k_s - \bar X^k_s)}^{2q} \le \frac{C_{\text{MZ},2q}}{N_1^{q}} \cV_{2q}(\bar\nu^X_s) \,,
\end{align*}
and similar to Step 2 in the proof of Theorem~\ref{th:main} (which is independent of this result),
\begin{align*}
    & \E \abs{\cM_\alpha(\bar\nu^X_s, \bar\nu^Y_s) - \cM_\alpha(\nu_{\bar\bX^{N_1}_s}, \bar\nu^Y_s)}^{2q} \le \\
    & \quad {C_{\text{MZ},2q} e^{2q\alpha (C_{\overline{\cE}} + C_{\underline{\cE}}) (1+R_\cut^2)}} (1+e^{2\alpha C_\cE(1+2R_\cut^2)})^q N_1^{-q} \cV_{2q}(\bar\nu^X_s) \,.
\end{align*}
Joining these and taking $\ell = \frac{C_{\text{rate,1}}(q) - q\kappa}{2(q-1)}$, we obtain
\begin{align*}
    & \E \left[ \sup_{s \in [0,t]} \abs{Z_s}^q \right] \le \\
    & (1+\ep^{-1})^q \frac{2^{3q-1} \sigma_1^{2q} (q-1)^q}{(C_{\text{rate,1}}(q) - q\kappa)^q} {C_{\text{MZ},2q} e^{2q\alpha (C_{\overline{\cE}} + C_{\underline{\cE}}) (1+R_\cut^2)}} (1+e^{2\alpha C_\cE(1+2R_\cut^2)})^q N_1^{-q} \cV_{2q}(\bar\nu^X_0) \,.
\end{align*}

We conclude as in the proof of Lemma~\ref{lm:tail-bound-particles}, with
\begin{align*}
    \bar C_{\text{tail,1}}(q) & = \frac{3^q}{2^q} C_{\text{tail,1}}(q) \\
    & \qquad + (1+\ep^{-1})^q \frac{2^{3q-1} \sigma_1^{2q} (q-1)^q}{(C_{\text{rate,1}}(q) - q\kappa)^q} {C_{\text{MZ},2q} e^{2q\alpha (C_{\overline{\cE}} + C_{\underline{\cE}}) (1+R_\cut^2)}} (1+e^{2\alpha C_\cE(1+2R_\cut^2)})^q \,.
\end{align*}
\end{proof}

\subsection{Technical lemmata}\label{s:tech}

The following lemmata build the connection between the usual mean and the weighted mean (consensus point) of probability measures.

\begin{lemma}
    \label{l:mean-consensus}
    For $\alpha > 0$, $p \ge 2$, $\mu_1 \in \cP_p(B_{d_1}(0,R_\cut))$ and $\mu_2 \in \cP_p(B_{d_2}(0,R_\cut))$, we have
    \begin{equation*}
        \abs{M(\mu_1) - \cM_\alpha(\mu_1,\mu_2)}^p \le e^{2\alpha C_\cE(1+2R_\cut^2)} \cV_p(\mu_1).
    \end{equation*}
\end{lemma}
\begin{proof}
    By Jensen's inequality, we have
    \begin{align*}
        \abs{M(\mu_1) - \cM_\alpha(\mu_1,\mu_2)}^p & \le \frac{\int \abs{x - M(\mu_1)}^p e^{-\alpha \cE(x, M(\mu_2))} \mu_1(dx)}{\int e^{-\alpha \cE(x, M(\mu_2))} \mu_1(dx)} \\
        & \le \frac{\int \abs{x - M(\mu_1)}^p e^{-\alpha \underline{\cE}(M(\mu_2))} \mu_1(dx)}{\int e^{-\alpha \overline{\cE}(x)} \mu_1(dx)} \\
        & = \frac{\cV_p{(\mu_1)}}{ e^{\alpha \underline{\cE}(M(\mu_2))} \ip{e^{-\alpha \overline{\cE}}, \mu_1} } \,.
    \end{align*}
    When subject to the compact searching domain, due to Condition~\ref{cd:efn}(4) we have
    \begin{equation*}
        \overline{\cE}(x) - \underline{\cE}(M(\mu_2)) \le 2C_\cE(1 + 2R_{cut}^2) \,.
    \end{equation*}
    Then 
    \begin{equation*}
        \abs{M(\mu_1) - \cM_\alpha(\mu_1,\mu_2)}^p \le e^{2\alpha C_\cE(1+2R_{cut}^2)} \cV_p(\mu_1) \,.
    \end{equation*}
\end{proof}

\begin{lemma}
    \label{l:decay-b}
    There exists some constant $C_M$, depending on $\alpha$, $R_\cut$, and $C_\cE$, such that
    \begin{align*}
        & \abs{ (\cM_\alpha(\mu_1, \mu_2) - M(\mu_1)) - (\cM_\alpha(\bar\mu_1, \bar\mu_2) - M(\bar\mu_1)) } \le \\
        & \qquad C_M (\sqrt{\var{\mu_1}} + \sqrt{\var{\bar\mu_1}}) (\cW_2(\mu_1, \bar\mu_1) + \cW_2(\mu_2, \bar\mu_2))
    \end{align*}
    for any $\mu_1, \bar\mu_1 \in \cP_2(B_{d_1}(0,R_\cut))$ and $\mu_2, \bar\mu_2 \in \cP_2(B_{d_2}(0,R_\cut))$.
\end{lemma}

\begin{proof}
    We first split the left-hand side by
    \begin{align}
        \nonumber
        & (\cM_\alpha(\mu_1, \mu_2) - M(\mu_1)) - (\cM_\alpha(\bar\mu_1, \bar\mu_2) - M(\bar\mu_1)) \\
        \label{e:split-Bt1}
        & = \frac{1}{Z} \left( \int x e^{-\alpha \cE(x, M(\mu_2))} \mu_1(dx) - Z M(\mu_1) - \int y e^{-\alpha \cE(y, M(\bar\mu_2))} \bar\mu_1(dy) + \bar Z M(\bar\mu_1) \right) \\
        \label{e:split-Bt2}
        & \qquad + (\frac{\bar Z}{Z} - 1 ) (\cM_\alpha (\bar\mu_1, \bar\mu_2) - M(\bar\mu_1)) \,,
    \end{align}
    where
    \begin{equation*}
        Z = \int e^{-\alpha \cE(x, M(\mu_2))} \mu_1(dx) \,, \qquad \bar Z = \int e^{-\alpha \cE(y, M(\bar\mu_2))} \bar\mu_1(dy) \,.
    \end{equation*}
    As $\cE$ is bounded on $B_{d_1}(0,R_\cut) \times B_{d_2}(0,R_\cut)$, we let $\cE_{\max}$ and $\cE_{\min}$ be some upper and lower bounds.
    Then $Z, \bar Z \ge e^{-\alpha \cE_{\max}} > 0$.

    We may reorganize the first term so that
    \begin{align*}
        & \int x e^{-\alpha \cE(x, M(\mu_2))} \mu_1(dx) - Z M(\mu_1) - \int y e^{-\alpha \cE(y, M(\bar\mu_2))} \bar\mu_1(dy) + \bar Z M(\bar\mu_1) \\
        & = \int (x - M(\bar\mu_1)) (e^{-\alpha \cE(x, M(\mu_2))} - Z) \mu_1(dx) - \int (y-M(\bar\mu_1)) (e^{-\alpha \cE(y, M(\bar\mu_2))} - \bar Z) \bar \mu_1(dy)  \\
        & = \iint (x-y) (e^{-\alpha \cE(x, M(\mu_2))} - Z) \pi(dx, dy) \\
        & \qquad + \iint (y-M(\bar\mu_1)) (e^{-\alpha \cE(x, M(\mu_2))} - e^{-\alpha \cE(y, M(\bar\mu_2))} - Z + \bar Z) \pi(dx, dy) \,,
    \end{align*}
    where $\pi \in \Pi(\mu_1, \bar\mu_1)$ is the coupling that attains $\cW_2(\mu_1, \bar\mu_1)$.
    Notice that $M$ is a linear operator.
    Applying Hölder's inequality to both integrals above, we have
    \begin{align*}
        & \abs{\int x e^{-\alpha \cE(x, M(\mu_2))} \mu_1(dx) - Z M(\mu_1) - \int y e^{-\alpha \cE(y, M(\bar\mu_2))} \bar\mu_1(dy) + \bar Z M(\bar\mu_1)} \\
        & \le \left( \iint \abs{x-y}^2 \pi(dx, dy) \right)^{\frac 12} \left( \iint \abs{e^{-\alpha \cE(x, M(\mu_2))} - Z}^2 \pi(dx,dy) \right)^{\frac 12} \\
        & \qquad + \left( \int \abs{y - M(\bar\mu_1)}^2 \bar\mu_1(dy) \right)^{\frac 12} \left( \iint \abs{e^{-\alpha \cE(x, M(\mu_2))} - e^{-\alpha \cE(y, M(\bar\mu_2))} - Z + \bar Z}^2 \pi(dx,dy) \right)^{\frac 12} \\
        & \le \cW_2(\mu_1, \bar\mu_1) \alpha e^{-\alpha \cE_{\min}} L_{\cE} \sqrt{\var{\mu_1}} + \sqrt{\var{\bar\mu_1}} \left( \alpha e^{-\alpha \cE_{\min}} L_\cE \cW_2(\mu_1, \bar\mu_1) + \abs{Z - \bar Z} \right) \,.
    \end{align*}
    Here we are using the Lipschitz continuity of $e^{-\alpha \cE}$ with Lipschitz constant $\alpha e^{-\alpha \cE_{\min}} L_\cE$, given the global Lipschitz continuity of $\cE$ itself.
    Moreover,
    \begin{align*}
        \abs{Z - \bar Z} & \le \iint \abs{e^{-\alpha \cE(x, M(\mu_2))} - e^{-\alpha \cE(y, M(\bar\mu_2))}} \pi(dx,dy)  \\
        & \le \alpha e^{-\alpha \cE_{\min}} L_{\cE} \iint \left(\abs{x-y} + \abs{M(\mu_2) - M(\bar\mu_2)} \right) \pi(dx,dy)  \\
        & \le \alpha e^{-\alpha \cE_{\min}} L_\cE (\cW_1(\mu_1, \bar\mu_1) + \cW_1(\mu_2, \bar\mu_2)) \,.
    \end{align*}
    Thus~\eqref{e:split-Bt1} is bounded by
    \begin{equation*}
        \alpha e^{\alpha (\cE_{\max} - \cE_{\min})} L_{\cE} \left[ \left(\sqrt{\var{\mu_1}} + 2 \sqrt{\var{\bar\mu_1}} \right) \cW_2(\mu_1, \bar\mu_1) + \sqrt{\var{\bar\mu_1}} \cW_2(\mu_2, \bar\mu_2) \right] \,.
    \end{equation*}

    For the second term~\eqref{e:split-Bt2}, we see that
    \begin{equation*}
        (\frac{\bar Z}{Z} - 1 ) (\cM_\alpha (\bar\mu_1, \bar\mu_2) - M(\bar\mu_1)) = \frac{\bar Z - Z}{Z \bar Z} \int (x - M(\bar\mu_1)) e^{-\alpha \cE(x, M(\bar\mu_2))}  \bar\mu_1(dx) \,.
    \end{equation*}
    Using the results from above, we have
    \begin{align*}
        & \abs{(\frac{\bar Z}{Z} - 1 ) (\cM_\alpha (\bar\mu_1, \bar\mu_2) - M(\bar\mu_1))} \\
        & \le e^{2\alpha \cE_{\max}} \abs{Z - \bar Z} \int \abs{x - M(\bar\mu_1)} e^{-\alpha \cE(x, M(\bar\mu_2))} \bar\mu_1(dx) \\
        & \le \alpha e^{2\alpha (\cE_{\max} - \cE_{\min})} L_{\cE} (\cW_1(\mu_1, \bar\mu_1) + \cW_1(\mu_2, \bar\mu_2)) \sqrt{\var{\bar\mu_1}} \,.
    \end{align*}

    Therefore, we reach the conclusion with $C_M = 3 \alpha e^{2\alpha (\cE_{\max} - \cE_{\min})} L_{\cE}$.
    By Conditions~\ref{cd:efn}(2)(4), we have $\cE_{\max} - \cE_{\min} \le 2C_\cE(1+R_\cut^2)$.
    This further gives $C_M = 3 \alpha e^{4\alpha C_\cE(1+R_\cut^2)} L_\cE$.
\end{proof}

\bibliographystyle{halpha-abbrv}
\bibliography{refs.bib}

\end{document}